\documentclass[11pt]{article}
\usepackage[utf8]{inputenc} 
\usepackage[T1]{fontenc} 
\usepackage{lmodern}  
\usepackage[a4paper]{geometry} 
\usepackage{amsmath, amssymb} 
\usepackage{amsthm} 
\usepackage{graphicx} 
\usepackage{xcolor} 
\usepackage{pgf,tikz}
\usetikzlibrary{arrows}
\usepackage[perpage]{footmisc} 
\usepackage{hyperref} 
\usepackage{array} 
\usepackage{url} 
\usepackage{verbatim}
\usepackage{multicol}
\usepackage{mdframed}
\usepackage{pstricks}
\usepackage{pstricks-add}
\usepackage{tabularx}
\usepackage{pgf,tikz}
\usetikzlibrary{arrows}
\usepackage{fancybox}
\usepackage{stmaryrd}
\usepackage{comment} 
\usepackage{caption}

\usepackage{pgfplots}
\pgfplotsset{compat=1.15}
\usepackage{mathrsfs}
\usetikzlibrary{arrows}

\usepackage{minitoc}
\mtcsetfeature{minitoc}{after}{\vspace{24pt}}

\usepackage{setspace}
\usepackage{enumitem}
\setlist[itemize]{label=\textbullet}

\theoremstyle{remark}
\newtheorem*{rmq}{Remark}
\newtheorem*{ex}{Example}
\newtheorem*{notation}{Notation}
\theoremstyle{plain}
\newtheorem*{rappelprop}{Proposition}
\newtheorem*{rappel}{Theorem}

\theoremstyle{plain}
\usepackage{thmbox}

\newtheorem{theorem}{Theorem}[section]
\newtheorem{prop}[theorem]{Proposition}
\newtheorem{lemma}[theorem]{Lemma}
\newtheorem{cor}[theorem]{Corollary}
\newtheorem{definition}[theorem]{Definition}

\usepackage[all, cmtip]{xy} 

\usepackage{fancyhdr}
\DeclareMathOperator{\Spec}{Spec}
\DeclareMathOperator{\Rspec}{R-Spec}

\DeclareMathOperator{\Rmax}{R-Max}
\DeclareMathOperator{\Max}{Max}
\DeclareMathOperator{\spm}{Spm}
\DeclareMathOperator{\F}{F}

\DeclareMathOperator{\Ftir}{F-}
\DeclareMathOperator{\fracloc}{Frac}

\DeclareMathOperator{\Id}{Id}

\DeclareMathOperator{\disc}{disc}

\DeclareMathOperator{\Rtir}{R-}
\DeclareMathOperator{\Cent}{Cent}
\DeclareMathOperator{\Dom}{Dom}
\DeclareMathOperator{\Indet}{Indet}

\newcommand{\m}{\mathfrak{m}}

\newcommand{\p}{\mathfrak{p}}
\newcommand{\q}{\mathfrak{q}}

\newcommand{\Ocal}{\mathcal{O}}
\newcommand{\Ccal}{\mathcal{C}}
\newcommand{\Zcal}{\mathcal{Z}}

\newcommand{\Vcal}{\mathcal{V}}
\newcommand{\Dcal}{\mathcal{D}}

\newcommand{\Rrm}{\mathrm{Rad}}
\newcommand{\inj}{\hookrightarrow}

\newcommand{\KO}{\mathcal{K}^0}
\newcommand{\RO}{\mathcal{R}^0}
\newcommand{\KR}{\mathcal{K}^{R+}}
\newcommand{\KP}{\mathcal{K}^{+}}
\newcommand{\K}{\mathcal{K}}
\newcommand{\piC}{\pi_{_\mathbb{C}}}
\newcommand{\piR}{\pi_{_\mathbb{R}}}
\newcommand{\C}{\mathbb{C}}
\newcommand{\R}{\mathbb{R}}
\newcommand{\A}{\mathbb{A}}
\newcommand{\RR}{\R[X^{R+}]}

\newcommand{\pC}{p_{_{\C}}}
\newcommand{\qC}{q_{_{\C}}}

\makeatletter
\renewcommand\tableofcontents{%
  \null\hfill\textbf{\Large\contentsname}\hfill\null\par
  \@mkboth{\MakeUppercase\contentsname}{\MakeUppercase\contentsname}%
  \@starttoc{toc}%
}
\makeatother

\author{FRANÇOIS BERNARD}
\title{\textbf{A notion of seminormalization for real algebraic varieties}}
\date{}
\begin{document}

\maketitle
\vspace{-0.6cm}
\begin{abstract}\noindent
    The seminormalization of an algebraic variety $X$ is the biggest variety linked to $X$ by a finite, birational and bijective morphism. In this paper we introduce a variant of the seminormalization, suited for real algebraic varieties, called the R-seminormalization. This object have a universal property of the same kind of the one of the seminormalization but related to the real closed points of the variety. In a previous paper, the author studied the seminormalization of complex algebraic varieties using rational functions that extend continuously to the closed points for the Euclidean topology. We adapt here some of those results to the R-seminormalization and we provide several examples. We also show that the R-seminormalization modifies the singularities of a real variety by normalizing the purely complex points and seminormalizing the real ones.

\end{abstract}

\makeatletter
\let\Hy@linktoc\Hy@linktoc@none
\makeatother

{\setlength{\baselineskip}{0.1\baselineskip}
\tableofcontents}

\makeatletter
\def\blfootnote{\gdef\@thefnmark{}\@footnotetext}
\makeatother


\vspace{-0.3cm}
\section*{Introduction.}

The present paper is devoted to the introduction of the R-seminormalization of real algebraic varieties. It can be seen as the real version of the previous paper \cite{Bernard2021} of the author about seminormalization and complex regulous functions. 

The operation of seminormalization was formally introduced around fifty years ago in the case of analytic spaces by Andreotti and Norguet \cite{Andre}. For algebraic varieties, the seminormalization $X^+$ of $X$ is the biggest intermediate variety between $X$ and its normalization which is bijective with $X$. Recently, the concept of seminormalization appears in the study of singularities of algebraic varieties, in particular in the minimal model program of Kollár and Kovács (see \cite{KollarSingularMMP} and \cite{KollarVariantsOfNormality}). The seminormalization has the property to have "multicross" singularities in codimension 1 (see \cite{Multicross}), it means that they are locally analytically isomorphic to the union of linear subspaces of affine space meeting transversally along a common linear subspace.\\
Around 1970 Traverso \cite{T} introduced the notion of the seminormalization $A^+_B$ of a commutative ring $A$ in an integral extension $B$. The idea is to glue together the prime ideals of $B$ lying over the same prime ideal of $A$. The seminormalization $A^+_B$ has the property that it is the biggest extension $C$ of $A$ in $B$ which is subintegral i.e. such that the map $\Spec(C)\to\Spec(A)$ is bijective and equiresidual (it gives isomorphisms between the residue fields). We refer to Vitulli \cite{VitulliSurvey2011} for a survey on seminormality for commutative rings and algebraic varieties.\\
In real algebraic geometry, the seminormalization is first studied in 1975 \cite{ABT1975} in the case of real analytic sets. However, in 1981, Marinari and Raimondo \cite{MarinariRaimondo81} show that the construction of Traverso with real spectrum has no natural universal property. Recently Fichou, Monnier and Quarez defined in \cite{FMQ} a real seminormalization called the "central" seminormalization but whose universal property does not rely on all the real closed points of the variety. The R-seminormalization is equipped with such a universal property and this is what motivated the introduction of this object.

The idea of the construction of the R-seminormalization is the following : since the seminormalization of a variety is obtained by gluing together the complex points that have been separated in the normalization, one may want to do the same thing but with real points of real varieties. Unfortunately, gluing the real points of the normalization in the fibers of the real points of a variety can lead to some problems because the restriction to the real closed points of a finite morphism of real varieties is not necessarily surjective. For example, the normalization of the real variety $X := \Spec(\R[x,y]/<y^2-x^2(x-1)>)$ is $X':=\Spec(\R[x,y]/<y^2-x+1>)$ and we have $(\pi')^{-1}(\{0\}) = \varnothing$ where $\pi'$ is the normalization morphism. In fact, for a general variety $X$, one can not find a variety $Y$ that would be maximal for the property of having a morphism $\pi : Y \to X$ which is finite, birational and such that $\pi_{\R}$ is bijective. A proof that such a variety does not exists in general can be found in \cite{FMQ} Example 5.6. This lack of surjectivity led the authors of \cite{FMQ} to consider the "central" seminormalization where they glue the real central points (i.e. the points in the Euclidean closure of the regular locus) of the normalization above the real central points of the variety. An other way to counter this lack of surjectivity is to consider all the complex points in the normalization lying over the real points of the variety. This is the idea of the R-seminormalization.

Recently, the present author highlighted in \cite{Bernard2021} the correlation between seminormalization and continuous rational functions for the Euclidean topology on complex affine varieties and on any algebraic variety over a field of characteristic 0 together with Fichou, Monnier and Quarez in \cite{BFMQ}. The spirit of the work presented here is to get an analog of the results of \cite{Bernard2021} for the R-seminormalization while getting a better understanding of this new object.

This paper is organized as follows. In Section \ref{SectionChap3:GeneralizationPU}, we give a generalization of Traverso's construction of the seminormalization. This allows us to define and to provide a universal property of the R-seminormalization. This universal property states that the R-seminormalization of a ring $A$ in an integral extension of $A$ is the biggest subextension such that there exists a unique prime ideal above each real prime ideal of $A$ and their residue fields are isomorphic. In Section \ref{SectionChap3:TheRSNGeometric}, we look at the R-seminormalization of real algebraic varieties and we prove a geometrical universal property for this object. We show that the R-seminormalization is the biggest variety between a variety $X$ and its normalization such that there exists a unique complex closed point above each element of $X(\R)$.
In Section \ref{SectionChap3:FunctionsOnTheRSN}, we identify by several ways, as made in the paper \cite{Bernard2021}, the ring of functions on $X(\R)$ which becomes polynomials on the closed points of the R-seminormalization. On the normalization, those functions correspond to the polynomials which are constant on the complex fibers over the real closed points of $X$. They also correspond to the integral rational functions which are (almost) continuous on $X(\R)$ for the topology of $X(\C)$. More precisely, we get the following version of Theorem 4.13 of \cite{Bernard2021} for the R-seminormalization.

\begin{rappelprop}[\ref{PropCritereSequentiel}]
Let $X$ be a real affine variety and let $f : X(\R) \to \R$. Then $f$ becomes polynomial on the R-seminormalization if and only if it verifies the following properties :
\begin{enumerate}
    \item The function $f$ is rational.
    \item The function $f$ is integral over $\R[X]$.
    \item For all $x\in X(\R)$, for all $(z_n)_n\in \Dom(f)\cup X(\R)$ such that $z_n \to x$ then $f(z_n) \to f(x)$.
\end{enumerate}
\end{rappelprop}
We give a third characterization of those functions using their graphs in the same spirit as Theorem 4.20 of \cite{Bernard2021}. In Section \ref{SectionChap3:RealKOcomplex}, we look at the functions on $X(\R)$ which are the restriction of continuous rational functions on $X(\C)$. This ring of functions correspond to the coordinate ring of the seminormalization $X^+$. Then we present several examples in order to compare the seminormalization, the R-seminormalization, the central seminormalization and the central weak-normalization.
Finally, in Section \ref{SectionLinkR-SNandBirNorm}, we prove that the R-seminormalization of a real variety $X$ is related to its seminormalization $X^+$ and to its biregular normalization $X^b$ which has been introduced by Fichou, Monnier and Quarez in \cite{IntClosure}. Briefly, the biregular normalization is the biggest variety which is linked to $X$ by a birational, finite and biregular morphism.
\begin{rappel}[\ref{TheoB+est+B}]
Let $X$ be a real affine variety. Then
$$ \R[X^{R+}] \simeq \R[X^+]^b \simeq \R[X^b]^+$$
\end{rappel}
This result allows us to see how the R-seminormalization modifies the singularities of a real algebraic variety : it normalizes the purely complex points and seminormalizes the real points.\\

\textbf{Acknowledgement :} This paper is part of Ph.D. Thesis of the author which was partly funded by the Centre Henri Lebesgue. The author is deeply grateful to G.Fichou and J-.P. Monnier for their precious help.

\section{Generalization of the seminormalization for general rings}\label{SectionChap3:GeneralizationPU}

In this section, we present a generalization of Traverso's construction of the seminormalization. The proofs of the results presented here being very similar to the one of Section 2 in \cite{Bernard2021}, we will refer to them a lot. This generalization will allow us to provide a universal property of the R-seminormalization but also of the R-Max-seminormalization which will be convenient in Section \ref{SectionChap3:TheRSNGeometric}.

Let $A$ be a ring. We say that an ideal $I$ of $A$ is real if it has the following property : If a sum of square elements $\sum a_i^2 \in \sum A^2$ belongs to $I$, then every element $a_i$ belongs to $I$. We denote by $\Rspec(A)$ (resp. $\Rmax(A)$) the set of real prime (resp. maximal) ideals of $A$. Let $\F$ be one of the functors $\Spec, \Max, \Rmax$ or $\Rspec$. In particular, $\F : \textbf{Ann} \to \textbf{Top}$ is a contravariant subfunctor of $\Spec$. It means that for all $A \in \textbf{Ann}$ we have $\F(A) \subset \Spec(A)$ and for all extension of rings $A\inj B$, we get a continuous application $\F(B) \to \F(A)$ given by $\p \to \p \cap A$. One may want to take a general subfunctor of $\Spec$ but we will need, at some point in this section, the specificity of the four considered subfunctors. See the remark after Definition \ref{DefFSubintegral} for more details.

\begin{definition}
Let $A \inj B$ be an integral extension of rings. We define
$$ A_B^{\F} := \{ b\in B \mid \forall \p\in \F(A) \text{,\hspace{0.2cm}} b_\p \in A_{\p}+\Rrm(B_{\p}) \}$$
\end{definition}

\begin{rmq}
If $\F  =\Spec$, the ring $A_B^{\F}$ is the seminormalization $A^+_B$ of $A$ in $B$. 
\end{rmq}

\begin{definition}\label{DefFSubintegral}
Let $A\inj B$ be an integral extension of rings. The extension $A \inj B$ is called $\F$-\textit{subintegral} if it verifies the following conditions :
\begin{enumerate}
    \item For all $\p\in \F(A)$, there exists a unique $\q \in \Spec(B)$ such that $\q\cap A = \p$.
    \item For such $\p$ and $\q$, we have $\q\in\F(B)$ and the induced map $\kappa(\p) \inj \kappa(\q)$ on the residue fields, is an isomorphism.
\end{enumerate}
\end{definition}

\begin{rmq}
It is important to see that, if $\p\in \F(A)$ and $\q \in \Spec(B)$ are such that $\q\cap A = \p$. The condition $\kappa(\p)\simeq\kappa(\q)$ implies that $\q\in\F(B)$. It is really specific to the fact that an ideal is real if and only if its residual field is a real field.
\end{rmq}

We give here a first geometric property of $\F$-subintegral extensions.

\begin{prop}\label{LemFiniBijDoncHomeoSurSpec}
Let $A \inj B$ be an $\F$-subintegral extension of rings and $\pi : \F(B) \to \F(A)$ be the induced map. Then $\pi$ is a Z-homeomorphism for the induced topology.
\end{prop}

\begin{proof}
The morphism $\Spec(B) \to \Spec(A)$ is Z-continuous, so is its restriction to $\F(B)$, thus we just have to show that $\pi$ is Z-closed. Let $\p\in \F(A)$ and $\q\in\F(B)$ such that $\q\cap A = \p$. We have
$$\pi(\mathcal{V}(\q)) = \{ \q'\cap A\mid \q\subset \q' \in \F(B) \}$$
so $\pi(\mathcal{V}(\q)) \subset \mathcal{V}(\q\cap A) = \mathcal{V}(\p)$. If $\p'\in \mathcal{V}(\p) \cap \F(A)$, then the going-up property says that there exists $\q'\in \Spec(B)$ such that $\q' \cap A = \p'$ and $\q\subset \q'$. Since the extension is $\F$-subintegral, the ideal $\q'$ belongs to $\F(B)$. Hence $\p'\in \pi(\mathcal{V}(\q))$ and finally $\mathcal{V}(\p) \subset \pi(\mathcal{V}(\q))$, so $\pi(\mathcal{V}(\q)) = \mathcal{V}(\p)$ which is Z-closed.
\end{proof}

The goal of this section is to prove the following universal property. It says that the $\F$-seminormalization of $A$ in $B$ is the biggest $\F$-subintegral extension of $A\inj B$.
\begin{prop}\label{PropPU}
Let $A\inj C\inj B$ be integral extensions of rings. Then the following statements are equivalent :
\begin{enumerate}
    \item[1)] The extension $A \inj C$ is $\F$-subintegral.
    \item[2)] The image of $C\inj B$ is a subring of $A_B^{\F}$.
\end{enumerate}
\end{prop}

We prove the universal property through a series of propositions. But first, let us recall the important property of "going-up" verified for integral extensions.
\begin{prop}[\cite{AM}]\label{LemGoingUp}
Let $A\inj B$ be an integral extension of rings. Then
\begin{enumerate}
    \item The associated map $\Spec(B) \to \Spec(A)$ is surjective.
    \item By the map $\Spec(B) \to \Spec(A)$, the inverse image of $\Max(A)$ is $\Max(B)$.
\end{enumerate}
\end{prop}

\begin{prop}\label{PropSnSub}
Let $A \inj B$ be an integral extension of rings. Then
$$A \inj A_B^{\F}\text{ is }\F\text{-subintegral}$$
\end{prop}

\begin{proof}
By doing the exact same proof as in Proposition 2.10 of \cite{Bernard2021}, one can show that for all $\p\in \F(A)$, there exists a unique $\q\in\Spec(A_B^{\F})$ such that $\q\cap A=\p$ and $\kappa(\p)\simeq\kappa(\q)$. Because of the special nature of $\F$ (see remark after Definition \ref{DefFSubintegral}), the equiresiduality implies that $\q\in\F(A_B^{\F})$ and so $A \inj A_B^{\F}$ is $\F$-subintegral.
\end{proof}

\begin{prop}\label{LemTransSub}
Let  $A \inj C \inj B$ be integral extensions of rings. Then the following properties are equivalent 
\begin{enumerate}
    \item[1)] The extension $A\inj B$ is $\Ftir$subintegral.
    \item[2)] The extensions $A\inj C$ and $C\inj B$ are $\Ftir$subintegral.
\end{enumerate}
\end{prop}

\begin{proof}
We prove 1) implies 2). Let $\p\in \F(C)$ and $\q_1,\q_2 \in \Spec(B)$ such that $\q_1\cap C = \q_2 \cap C = \p$. We have $\q_1 \cap A = \q_1 \cap C \cap A = \p \cap A$ and the same is true for $\q_2$. Since $A\inj B$ is $\Ftir$subintegral and $\p\cap A\in\F(A)$, then $$\q_1 \cap A = \q_2 \cap A \implies \q_1 = \q_2$$
Moreover $\kappa(\p\cap A)\inj \kappa(\p)\inj \kappa(\q_i)$ and $\kappa(\p\cap A) \simeq \kappa(\q_i)$. So $\kappa(\p)\simeq \kappa(\q_i)$. Now, let us consider $\p\in\F(A)$ and $\p_1,\p_2 \in \Spec(C)$ such that $\p_1 \cap A = \p_2 \cap A = \p$. We know that $\Spec(B)\to\Spec(C)$ is surjective, so we can find $\q_1,\q_2 \in \Spec(B)$ such that $\q_1\cap C = \p_1$ and $\q_2\cap C = \p_2$. Then $\q_1\cap A = \p$ and $\q_2\cap A = \p$. Since $A\inj B$ is $\Ftir$subintegral and $\p\in\F(A)$, we get $\q_1 = \q_2$ and so $\p_1=\p_2$. Moreover $\kappa(\p)\inj \kappa(\p_i)\inj \kappa(\q_i)$ and $\kappa(\p) \simeq \kappa(\q_i)$. So $\kappa(\p)\simeq \kappa(\p_i)$.

We now show that 2) implies 1). Let's suppose that $A\inj C$ and $C \inj B$ are $\Ftir$subintegral. Let $\p\in \F(A)$, then there exists a unique element $\p'\in\Spec(C)$ such that $\p'\cap A = \p$ and $\p'\in\F(C)$. Moreover $\kappa(\p')\simeq\kappa(\p)$. Then, since $C \inj B$ is $\Ftir$subintegral, there exists a unique element $\p'' \in \Spec(B)$ such that $\p'' \cap C = \p'$ and $\kappa(\p'')\simeq\kappa(\p')$. So $\p''$ is the unique element of $\Spec(B)$ such that $\p''\cap A = \p$ and $\p''\in\F(B)$. Moreover $\kappa(\p'')\simeq\kappa(\p)$. Hence $A\inj B$ is $\Ftir$subintegral.
\end{proof}

\begin{rappelprop}[Proposition \ref{PropPU}]
Let $A\inj C\inj B$ be integral extensions of rings. Then the following statements are equivalent :\begin{enumerate}
    \item[1)] The extension $A \inj C$ is $\Ftir$subintegral.
    \item[2)] The image of $C\inj B$ is a subring of $A^{\F}_B$.
\end{enumerate}
\end{rappelprop}

\begin{proof}
One can show 1) implies 2) by doing the exact same proof of Proposition 2.4 of \cite{Bernard2021} with $A_B^{\F}$ instead of $A^+_B$. The converse is also very similar : suppose that we have $A\inj C \inj A_B^{\F} \inj B$. Those extensions are integral and, by Proposition \ref{PropSnSub}, the extension $A\inj A_B^{\F}$ is $\Ftir$subintegral. Then Lemma \ref{LemTransSub} tells us that $A\inj C$ is $\Ftir$subintegral.
\end{proof}

Let us conclude this section by rewriting Proposition \ref{PropPU} in the form of a universal property theorem.

\begin{theorem}[Universal property of the $\Ftir$seminormalization]
Let $A\inj B$ be an integral extension of rings. For every intermediate extension $C$ of $A\inj B$ such that $A\inj C$ is $\Ftir$subintegral, the image of $C$ by the injection $C\inj B$ is contained in $A_B^{\F}$.

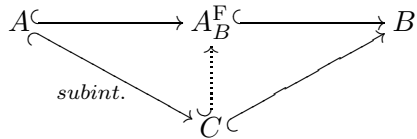
$$\xymatrix{
   A \ar@{_{(}->}[rrd]_{subint.} \ar@{^{(}->}[rr]& & A_B^{\F} \ar@{^{(}->}[rr] & & B \\
    && C \ar@{^{(}.>}[u] \ar@{_{(}->}[urr] &&
}$$
\end{theorem}

\begin{rmq}
Let $A\inj B$ be an integral extension. We have that $A \inj A_B^{\F}$ is $\Ftir$subintegral by Proposition \ref{PropSnSub}. So we can apply the universal property in the following way :
$$\xymatrix{
   A \ar@{_{(}->}[rrd]_{subint.} \ar@{^{(}->}[rr]& & A_{A_B^{\F}}^{\F} \ar@{^{(}->}[rr] & & A_B^{\F} \\
    && A_B^{\F} \ar@{^{(}.>}[u] \ar@{_{(}->}[urr] &&
}$$
Thus $A_B^{\F} \inj A_{A_B^{\F}}^{\F}$. But, by definition, $A_{A_B^{\F}}^{\F}$ is included in $A_B^{\F}$. We get the following idempotency property $$A_B^{\F} = A_{A_B^{\F}}^{\F}$$
\end{rmq}

\section{The R-seminormalization for geometric rings}\label{SectionChap3:TheRSNGeometric}

Let $X = \Spec(A)$ be an affine algebraic variety with $A$ a $\R$-algebra of finite type. Let $\R[X] := A$ denote the coordinate ring of X. We have $\R[X] \simeq \R[x_1, ..., x_n]/I$ for an ideal $I$ of $\R[x_1, ..., x_n]$. We will say that $X$ is a real affine variety if $I$ is a real ideal. A morphism $\pi : Y \to X$ between two real varieties induces the morphism $\pi^* : \R[X] \inj \R[Y]$ which is injective if and only if $\pi$ is dominant. We say that $\pi$ is of finite type (resp. is finite) if $\pi^*$ makes $\R[Y]$ a $\R[X]$-algebra of finite type (resp. a finite $\R[X]$-module). The space $X$ is equipped with the Zariski topology for which the closed sets are of the form $\Vcal(I) := \{\p \in \Spec(\R[X]) \mid I \subset \p\}$ where $I$ is an ideal of $\R[X]$. We define $X(\R) := \{\m \in \Max(\R[X]) \mid \kappa(\m) \simeq \R\} = \Rmax(\R[X])$. The Real Nullstellensatz gives us a Zariski homeomorphism between $X(\R)$ and the algebraic set $\Zcal_{\R}(I) := \{x\in \R^n \mid \forall f \in I \text{ }f(x) = 0\} \subset \R^n$. For any real variety, we can look at its complexification whose coordinate ring is given by the change of coordinate $\C[X] := \R[X]\otimes_{\R}\C$. Hence we can consider the set of its closed points $X(\C)$ and if $X$ is a real variety, we have that $X(\R)$ is Z-dense in $X(\C)$. We note $\pi_{\R} : Y(\R) \to X(\R)$ the restriction of $\pi$ to $Y(\R)$ and $\pi_{\C}$ the restriction of $\pi$ to $Y(\C)$. If X is irreducible, then we write $\K(X) := \fracloc(\R[X])$.

The goals of this section are to provide a geometric universal property of the R-seminormalization and to see that the R-seminormalization coincide with the R-Max-seminormalization for real algebraic varieties. The following theorem gives a reinterpretation of the R-subintegrality in a geometric point of view.

\begin{theorem}\label{TheoRSubEqBij}
Let $\pi : Y \to X$ be a finite morphism between real affine varieties. The following properties are equivalent
\begin{enumerate}
    \item[1)] The extension $\pi^* : \R[X] \inj \R[Y]$ is $\Rtir$subintegral.
    \item[2)] The extension $\pi^* : \R[X] \inj \R[Y]$ is $\Rmax$-subintegral.
    \item[3)] The restriction $\widetilde{\pi} : \piC^{-1}(X(\R)) \to X(\R)$ of the morphism $\piC : Y(\C) \to X(\C)$ is bijective.
    \item[4)] The morphism $\piR : Y(\R) \to X(\R)$ is bijective and $\piC^{-1}(X(\R)) = Y(\R)$.
\end{enumerate}
\end{theorem}

\begin{proof}
1) $\implies$ 2). Suppose that $\R[X] \inj \R[Y]$ is $\Rtir$subintegral, then for all $\p\in\Rspec(\R[X])$, there exists a unique $\q\in\Spec(\R[Y])$ such that $\q\cap \R[X] = \p$. Moreover $\q\in\Rspec(\R[Y])$. In particular, it is true if $\p$ is maximal and so we get Property $2$.

2) $\implies$ 1). Suppose that for all $\m\in \Rmax(\R[X])$, we have a unique element $\m' \in \spm(\R[Y])$ such that $\m'\cap \R[X] = \m$ and $\m'\in\Rmax(\R[Y])$. Let $\p\in\Rspec(\R[X])$ and $\q\in \Spec(\R[Y])$ such that $\q\cap \R[X] = \p$. We note $V = \Vcal(\p)$ and $W = \Vcal(\q)$. 
Then $\dim V = \dim W$ since $\pi$ is finite. Moreover $\dim V(\R) = \dim V(\C)$ because $\p$ is a real ideal and we also have $\dim V(\R) = \dim W(\R)$ since $\piR$ is bijective. We get 
$$ \dim W(\R) = \dim V(\R) = \dim V(\C) = \dim V = \dim W = \dim W(\C) $$ 
So $\q \in \Rspec(\R[Y])$. Let us show that $\q$ is unique. Suppose there exists $\q'\in \Spec(\R[Y])$ such that $\q'\cap \R[X] = \p$. By the previous arguments we have $\q' \in \Rspec(\R[Y])$. Let $\m \in \Rmax(\R[Y])$ such that $\q \subset \m$ then $\q\cap \R[X] \subset \m\cap\R[X]$. Moreover, by the going-up property, we can consider $\m' \in \spm(\R[Y])$ such that $\q'\subset \m'$ and $\m'\cap \R[X] = \m\cap \R[X]$. Then, by assumption, we get $\m = \m'$. So $\Zcal_{\R}(\q) = \Zcal_{\R}(\q')$ and the real Nullstellensatz gives us $\q = \q'$.

It remains to see that $\kappa(\p') \simeq \kappa(\p)$. If we write $V = \Spec(\R[X]/\p)$ and $W = \Spec(\R[Y]/\p')$, we get the following commutative diagram :

$$\xymatrix{
    \R[X] \ar@{->>}[d]^{\pi_X} \ar@{^{(}->}[r]^{\pi^*}& \R[Y] \ar@{->>}[d]^{\pi_Y} \\
    \R[V] \ar@{^{(}->}[d] \ar@{^{(}->}[r]^{(\pi_{|_W})^*} &\R[W] \ar@{^{(}->}[d] \\
     \K(V) \ar@{^{(}->}[r] & \K(W)
}$$

As $\R[Y]$ is a finite $\R[X]$-module, we have that $\R[W]$ is a finite $\R[V]$-module. Thus $\pi_{|W}$ is a finite morphism between two irreducible varieties. Therefore we get $n:= [K(W):K(V)] < +\infty$. Since the characteristic is zero, the extension $K(V)\inj K(W)$ is separable and finite. Hence we can consider a primitive element $a\in K(W)$ such that $K(W) = K(V)(a)$. Let $F$ be the minimal polynomial of $a$. Then $\deg F = n$ and $\disc(F) \neq 0$ where $\disc(F)$ is the discriminant of $F$. So there exists a Z-open set of $V$ such that $\disc(F) \neq 0$. Since $\p\in\Rspec(\R[X])$, we have that $V(\R)$ is Z-dense in $V$ and so there exists $y\in V(\R)$ such that $\disc(F)(y,.) \neq 0$. This means 
$$ \#\piC^{-1}(y) = \#\{ \text{ complex roots of }F(y,.) \} = \deg F = n  $$
By assumption, there is a unique element of $Y(\C)$ above every element of $X(\R)$. So $\#\piC^{-1}(y) = n = 1$ and so $\kappa(\p) \simeq \kappa(\p')$.\\

Now, let us see that the properties $2,3$ and $4$ are equivalent.

$2) \implies 3)$. Property $2$ implies that for all $\m\in\Rmax(\R[X])$, there exists a unique $\m'\in\Spec(\R[Y])$ such that $\m'\cap\R[X] = \m$. By the going-up property, we have $\m'\in\spm(\R[Y])$. So, by the Nullstellensatz, we get the third property.

$3) \implies 4)$. Let $x\in X(\R)$ and $z\in Y(\C)$ such that $\piC(z) = x$, then $\piC(\overline{z})=x$. Since $z$ is supposed to be unique, we get $z=\overline{z}$ and so $z\in Y(\R)$.

$4) \implies 2)$. Suppose $4$ and take $\m\in \Rmax(\R[X])$. By the going-up property, we know that there is a finite number of prime ideals in $\R[Y]$ lying over $\m$ and that those ideals are maximal. Since we have $\piC^{-1}(X(\R)) = Y(\R)$, then all the $\m'$ are real. So, the morphism $\piR$ being bijective, we get that there is a unique prime ideal $\m'$ of $\R[Y]$ lying over $\m$. Moreover $\m'\in\Rmax(\R[Y])$ so $\kappa(\m) \simeq \kappa(\m') \simeq \R$ and we get the second property.

\end{proof}

Let $A\inj B$ be an integral extension of rings, we note \begin{center}
    $A^{\Rspec}_B = A_B^{R+}$ and $A^{\Rmax}_B = A^{R+_{\max}}_B$
\end{center}

Let us see that, if $A$ is a coordinate ring, then $A^{R+}$ and $A^{R+_{\max}}$ are also the coordinate rings of some real varieties. 

\begin{prop}
Let $\pi : Y\to X$ be a finite morphism between real affine varieties. Let $A$ be an intermediate ring between $\R[X]$ and $\R[Y]$. Then there exists a unique affine variety $Z$ such that $A = \R[Z]$. Moreover, if $X$ and $Y$ are real varieties and $\pi$ is birational, then $Z$ is a real affine variety.
\end{prop}

\begin{proof}
We have that $A$ is a $\R[X]$-module because it is a submodule of $\R[Y]$. Thus it is a $\R$-algebra of finite type because so is $\R[X]$. If $\pi$ is also birational, then we get $\R[X] \inj \R[Z]\inj \R[X]'$ and by \cite{IntClosure} Lemma 2.8, the ring $\R[Z]$ is real.
\end{proof}

This lead us to define, for every real variety, a new variety called its R-seminormalization.
\begin{definition}\label{DefvarieteRSN}
Let $\pi : Y \to X$ be a finite morphism between two affine varieties over $\R$. The affine variety defined by 
$$ X^{R+}_Y := \Spec(\R[X]^{R+}_{\R[Y]}) $$
is called the $\Rtir$seminormalization of $X$ in $Y$.
\end{definition}

The $\Rtir$subintegrality being equivalent the $\Rmax$-subintegrality for affine rings, we naturally get that the $\Rmax$-seminormalization correspond to the $\Rtir$seminormalization. Note that, for the central seminormalization defined in \cite{FMQ}, this property is not true and we get two different varieties : the central seminormalization $X^{s_C}$ and the central weak-normalization $X^{w_C}$.

\begin{cor}\label{CorSemiMaxEstSemi}
Let $\pi : Y \to X$ be a finite morphism between two real affine varieties. Then
$$ \R[X]^{R+_{\max}}_{\R[Y]} = \R[X]^{R+}_{\R[Y]} $$
\end{cor}

\begin{proof}
First, the inclusion $\R[X]^{R+}_{\R[Y]} \subset \R[X]^{R+_{\max}}_{\R[Y]}$ is clear. Now, by Proposition \ref{PropSnSub}, we know that $\R[X]^{R+_{\max}}_{\R[Y]}$ is $\Rmax$-subintegral, so by Theorem \ref{TheoRSubEqBij}, it is also $\Rtir$subintegral. Then the universal property of $\R[X]^{R+}_{\R[Y]}$ gives us $\R[X]^{R+_{\max}}_{\R[Y]} \subset \R[X]^{R+}_{\R[Y]}$.
\end{proof}

We can now rewrite the universal property of the $\Rtir$seminormalization for the geometric case. It is the biggest variety such that there is a unique complex closed point in the fiber of every real closed point.

\begin{theorem}[Universal property of the $\Rtir$seminormalization]\label{TheoPUdelaRSN}
Let $Y\to Z\to X$ be finite morphisms between real affine varieties. Then the restriction $\widetilde{\pi_Z^{}} : \pi_{Z(\C)}^{-1}(X(\R)) \to X(\R)$ of the morphism $\pi_{Z(\C)} : Z(\C)\to X(\C)$ is bijective if and only if there exits a morphism $\pi_Z^+ : X^{R+}_Y \to Z$ such that $\pi_Z^{}\circ\pi^+_Z = \pi^+$. Moreover $\pi^+_Z$ is unique and the restriction $\widetilde{\pi^+_Z} : (\pi^+_Z)^{-1}(Z(\R)) \to Z(\R)$ of the morphism $\pi^+_Z : X^{R+}_Y(\C) \to Z(\C)$ is bijective.

$$
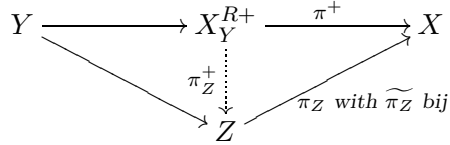
\xymatrix{
    Y \ar[rrd] \ar[rr]&& X^{R+}_Y \ar@{.>}[d]_{\pi^+_Z} \ar[rr]^{\pi^+} && X \\
    && Z \ar[urr]_{\pi_Z \text{ with }\widetilde{\pi_Z}\text{ bij}} &&
}$$

\end{theorem}

\begin{proof}
We have the following equivalences :
$$\begin{array}{rcl}
    \text{The morphism }\widetilde{\pi_Z^{}} \text{ is bijective} & \iff & \R[X] \inj \R[Z] \text{ is } \Rtir\text{subintegral} \text{ ( by Theorem \ref{TheoRSubEqBij} )}\\
     & \iff & \R[Z] \inj \R[X_{\R[Y]}^{R+}] \text{ ( by Proposition \ref{PropPU} )}\\
     & \iff & \exists \pi_Z^+ : X_Y^{R+} \to Z \text{ dominant, such that } \pi_Z^{} \circ \pi^+_Z = \pi^+
\end{array}$$
We get the uniqueness of $\pi_Z^+$ by injectivity of $\pi_Z$. Since $\R[X] \inj \R[X^{R+}_Y]$ is $\Rtir$subintegral then Lemma \ref{LemTransSub} says that $\R[Z] \inj \R[X^{R+}_Y]$ is also $\Rtir$subintegral. So, by Theorem \ref{TheoRSubEqBij}, the morphism $\widetilde{\pi_Z^+}$ is bijective.
\end{proof}

\section{Functions which become polynomial on the R-seminormalization}\label{SectionChap3:FunctionsOnTheRSN}


Let $X$ be a real variety. The normalization $X'$ of a real variety is defined by the coordinate ring $\R[X]'_{\K(X)}$ and its seminormalization $X^+$ is defined by the coordinate ring $\R[X]^+_{\R[X']}$. Moreover, we have $\C[X]' = \R[X']\otimes_{\R} \C$ and $\C[X]^+ = \R[X^+]\otimes_{\R} \C$.\\

For a reduced affine variety $X$ with a finite number of irreducible components, the total ring of fractions $\K(X)$ is a product of fields. Moreover, in order to look at the normalization of $\R[X]$ in $\K(X)$, one can look at the normalization of each irreducible component in its field of fractions.\\

The R-seminormalization $X^{R+}$ of $X$ is the variety $X^{R+}_{X'}$ defined in Definition \ref{DefvarieteRSN}. The R-seminormalization $X^{R+}$ comes with a finite, birational morphism $\pi^{R+} : X^{R+} \to X$ such that $\piR^{R+}$ is bijective and $(\piC^{R+})^{-1}(X(\R)) = X^{R+}(\R)$. Its universal property is given by Theorem \ref{TheoPUdelaRSN}.\\

We have the following extensions of rings :
$$ \R[X] \inj \R[X^+] \inj \R[X^{R+}] \inj \R[X'] \inj \K(X). $$ 

In all this section, we will consider real varieties with real irreducible components. The set of real closed points $X(\R)$ can be seen as a subset of $\R^n$ for some $n$. Hence we will be able to use the Euclidean topology on $X(\R)$ induced by $\R^n$.\\

In the same spirit as in \cite{Bernard2021}, we want to study the set $\KR(X(\R))$ of functions defined on $X(\R)$ which become polynomials on the R-seminormalization. This can be useful, for example, to construct the R-seminormalization of a given real variety. We start by identifying the elements of $\R[X']$ which comes from an element of $\R[X^{R+}]$. Then we show that the elements of $\KR(X(\R))$ are integral rational functions which are almost continuous on $X(\R)$ for the topology of $X(\C)$. Finally, we provide a last characterization of the elements of $\KR(X(\R))$ with properties concerning their graph.\\

\begin{definition}
Let $X$ be a real variety and $\pi_{\R}^{R+}$ be its R-seminormalization morphism. We will denote by $\KR(X(\R))$ the ring of functions $f : X(\R) \to \R$ such that $f\circ\pi_{\R}^{R+} \in \R[X^{R+}]$.
\end{definition}

We start by identifying the elements of $\R[X^{R+}]$ in $\R[X']$. Note that this identification is interesting but does not identify the ring $\KR(X(\R))$ independently from the variety $X^{R+}$. 

\begin{prop}\label{PropConstanteSurLesFibres}
Let $X$ be a real affine variety and let $\pi':X'\to X$ be the normalization morphism of $X$. Then
$$ \R[X^{R+}] \simeq \{ p\in\R[X'] \mid \forall z_1,z_2 \in \piC'^{-1}(X(\R)) \text{ , } \piC'(z_1) = \piC'(z_2) \implies p_{_{\C}}(z_1) = p_{_{\C}}(z_2) \}$$
\end{prop}

\begin{proof}
We will denote by $\pi : X' \to X^{R+}$ the morphism induced by the extension $\R[X^{R+}]\inj\R[X']$. Let $q \in \RR$ and $x \in X(\R)$. We want to show that for all $z_1,z_2 \in \piC'^{-1}(x)$, we have $q\circ\piC(z_1) = q\circ\piC(z_2)$.
We have 

$$\begin{array}{ccc}
\R[X]_{\m_x} &\inj & \R[X']_{\m_x} \smallskip \\
\m_x\R[X]_{\m_x} &\mapsfrom& \m_1 \R[X']_{\m_x}\\
&&\hspace{0.5cm}\vdots\\
&&\m_n \R[X']_{\m_x}
\end{array}$$

where the $\m_i \in \Max\R[X']$ are such that $\m_i\cap \R[X] = \m_x$. By definition of $\RR$, we can write 
$(q\circ\pi_{\R})_x = \alpha\circ\pi'_{\R}+\beta$ with $\alpha\in\R[X]_x$ and $\beta \in\Rrm(\R[X']_x)$. Then 
$$\forall z\in\piC'^{-1}(\{x\})\text{ \hspace{0.6cm}}(q\circ\piC)_x(z) = \alpha\circ\piC'(z)+\beta_{_{\C}}(z) = \alpha_{_{\C}}\circ\piC'(z)+\beta_{_{\C}}(z) = \alpha_{_{\C}}(x)+\beta_{_{\C}}(z)$$

Moreover, if we write $\piC'^{-1}(\{x\}) = \{ x_1, ..., x_k, z_{k+1}, \overline{z_{k+1}} ..., z_n, \overline{z_n} \}$, then we have 
$$\beta_{_{\C}} \in \bigcap_{i=1}^n \m_i\C[X']_{\m_x} = \bigcap_{i=1}^k \m_{x_i}\C[X']_{\m_x} \cap \bigcap_{i=k+1}^n \m_{z_i}\cap \m_{\overline{z_i}}\C[X']_{\m_x}$$

So for all $z\in \piC'^{-1}(\{x\})$, we get $\beta_{\C}(z) = 0$ and finally
$$ \forall z\in \piC'^{-1}(\{x\}) \text{ \hspace{0.6cm}} q\circ\piC(z) = \alpha(x)$$

Conversely, let $p\in\R[X']$, $x\in X(\R)$ and $\piC'^{-1}(\{x\}) = \{ x_1, ..., x_k, z_{k+1}, \overline{z_{k+1}} ..., z_n, \overline{z_n} \}$ with $x_i\in X'(\R)$ and $z_i\in X'(\C)\setminus X'(\R)$. Suppose there exists $\alpha\in\C$ such that, for all $z\in\piC'^{-1}(\{x\})$, we have $p_{\C}(z) = \alpha$. Since $p(z) = p(\overline{z})=\overline{p(z)}=\alpha$, we get that $\alpha\in\R$ and so $\alpha_x\in\R[X]_{\m_x}$. Then we can write 
$$ p_x = \alpha_x + (p_x - \alpha_x) \in \R[X']_{\m_x}$$
Moreover, for all $z \in \piC^{-1}(x)$, we have $(p-\alpha)_{\C} \in \m_z\C[X']$. So
$$\begin{array}{cl}
     (p_x - \alpha_x)_{\C} \in & \R[X']_{\m_x}\cap\displaystyle{\bigcap_{z\in \piC'^{-1}(\{x\})} \m_z\C[X']_{\m_x}} \\ [0.8cm]
     & = \displaystyle{\R[X']_{\m_x}\cap\bigcap_{i=1}^k \m_{x_i}\C[X']_{\m_x} \cap \bigcap_{i=k+1}^n \m_{z_i}\cap \m_{\overline{z_i}}\C[X']_{\m_x}} \\ [0.7cm]
     & = \displaystyle{\bigcap_{i=1}^k \m_{x_i}\R[X']_{\m_x} \cap \bigcap_{i=k+1}^n ( \m_{z_i}\cap \overline{\m_{z_i}})\R[X']_{\m_x}} \\ [0.7cm]
     & = \displaystyle{\bigcap_{\m\cap\R[X] = \m_x} \m\R[X'] } \\ [0.6cm]
     & = \Rrm(\R[X']_{\m_x})
\end{array}$$
So $p_x - \alpha_x \in \Rrm(\R[X']_{\m_x})$.
\end{proof}

\begin{cor}
Let $X$ be a real affine variety and $\pi' : X' \to X$ be the normalization morphism. Then
$$ \KR(X(\R)) \simeq \Bigl\{ p\in\R[X'] \mid \forall z_1,z_2 \in \piC'^{-1}(X(\R)) \text{ , } \piC'(z_1) = \piC'(z_2) \implies p_{_{\C}}(z_1) = p_{_{\C}}(z_2) \Bigr\} $$
\end{cor}

\begin{rmq}
Note that if $f\in\KR(X(\R))$ and $p$ is the element of $\R[X']$ above $f$, then for all $x\in X(\R)$ and for all $z\in \piC'^{-1}(x)$, we have $p_{_{\C}}(z) = f(x) \in \R$.
\end{rmq}

In the paper \cite{FMQ}, the authors introduced the central seminormalization ( resp. weak-normalization ) of real algebraic varieties. The idea of those constructions is to glue together the central points of the normalization over the central points of the variety (resp. over the maximal central points). The central locus of a real variety being the Euclidean closure of its real regular locus. They have shown that, for a real algebraic variety $X$, the coordinate ring of its central weak-normalization $X^{w_c}$ corresponds to the integral closure of $\R[X]$ in the ring $\KO(\Cent(X))$ of continuous rational functions on the central points of $X(\R)$. Also, the coordinate ring of its central seminormalization $X^{s_c}$ correspond to the integral closure of $\R[X]$ in the ring $\RO(\Cent(X))$ of regulous functions on the central points of $X(\R)$. Those are the continuous rational functions which stay rational by restriction to the real closed points of a real subvariety of $X$. The elements of $\KO(X(\R))$ and $\RO(X(\R))$ have been extensively studied in real algebraic geometry. One can found more information about regulous and rational continuous functions in the survey of Kucharz and Kurdyka \cite{KucharzKurdyka2018}.

The following lemmas \ref{LemR+Condition1et2}, \ref{LemHOestContinue} and \ref{LemHOestRegulue} shows that the elements of $\KR(X(\R))$ are regulous integral functions and so that they become polynomial on $\R[X^{s_c}]$ and $\R[X^{w_c}]$.

\begin{lemma}\label{LemR+Condition1et2}
Let $X$ be a real affine variety and $f \in \KR(X(\R))$. Then $f$ is rational and there exists a monic polynomial $P(t)\in \C[X][t]$ such that $P(f)=0$ on $X(\R)$.
\end{lemma}

\begin{proof}
Let $f:X(\R)\to\R$ be such that $f\circ \piR^{R+} \in \RR$. First we have $f\in \K(X)$ because $\pi^{R+}$ is birational. Now, since $f\circ \piR^{R+} \in \RR \implies f\circ \piR' \in \R[X']$, then there exists a monic polynomial $P(t) \in \R[X][t]$ such that $$P(f\circ\piR') = (f\circ\piR')^n+(a_{n-1}\circ\piR')(f\circ\piR')^{n-1}+...+(a_0\circ\piR') = 0$$ 
Since we have an injection $\R[X^{R+}]\inj\R[X']$, we get 
$$(f\circ\piR^{R+})^n+(a_{n-1}\circ\piR^{R+})(f\circ\piR^{R+})^{n-1}+...+(a_0\circ\piR^{R+}) = 0$$
Now, by Proposition \ref{LemGoingUp}, we have $\piR^{R+}$ surjective. So 
$$f^n+a_{n-1}f^{n-1}+...+a_0 = 0$$
So $f$ is integral over $\R[X]$.
\end{proof}

\begin{lemma}\label{LemHOestContinue}
Let $X$ be a real affine variety and $f \in \KR(X(\R))$. Then
\begin{center}$f$ is continuous for the Euclidean topology on $X(\R)$.\end{center}
\end{lemma}

\begin{proof}
Let $f\in \KR(X(\R))$ and $F$ be an Euclidean closed set of $\R$. Since $f\circ\piR^{R+} \in \R[X^{R+}]$, we have that 
$ (\piR^{R+})^{-1}(f^{-1}(F)) = (f\circ\piR^{R+})^{-1}(F) \text{ is closed.} $
By \cite{FMQ} Lemma 3.1, the function $\piR^{R+}$ is closed for the Euclidean topology. Then
$$ \piR^{R+}\Bigl( (\piR^{R+})^{-1}(f^{-1}(F)) \Bigr) = f^{-1}(F) \text{ is closed.}$$
and so $f$ is continuous.
\end{proof}

\begin{lemma}\label{LemHOestRegulue}
Let $X$ be a real affine variety, $V\subset X$ be a real subvariety of $X$ and $f \in \KR(X(\R))$. Then $$f_{|V(\R)}\text{ is rational}$$
\end{lemma}

\begin{proof}
Let $V$ be a subvariety of $X$. Then there exists $\p \in \Rspec(\R[X])$ such that $\R[V] \simeq \R[X]/\p$. So there is a unique $\q \in \Spec(\RR)$ such that $\q\cap\R[X] = \p$. We note $W$ the subvariety of $X^{R+}$ such that $\R[W]\simeq \RR / \q$. Then we get the following commutative diagram.

$$\xymatrix{
    \R[X] \ar@{->>}[d] \ar@{^{(}->}[r]^{\left(\pi^{R+}\right)^*}& \R[X^{R+}] \ar@{->>}[d] \\
    \R[V] \ar@{^{(}->}[d] \ar@{^{(}->}[r]^{\left(\pi_{|W}^{R+}\right)^*} &\R[W] \ar@{^{(}->}[d] \\
     \K(V) \ar@{=}[r] & \K(W),
}$$
We have $f\circ\piR^{R+} \in \RR$ so $f_{|V(\R)}\circ\pi_{|W(R)}^{R+} = (f\circ\pi^{R+})_{|W(\R)} \in \R[W]$. Then $\frac{ f_{|V(\R)}\circ\pi_{|W(R)}^{R+} }{1} \in \K(W)$ and so $f_{|V(\R)}\in\K(V)$.
\end{proof}

\begin{rmq}
We have shown that the elements of $\KR(X(\R))$ are regulous functions and that they are integral over $\R[X]$. So $$ \KR(X(\R)) \subset \R[X]'_{\RO(X(\R))} \subset \R[X]'_{\KO(X(\R))} $$
\end{rmq}

In the same spirit as for $X^{s_c}$ and $X^{w_c}$, one may wonder if the ring $\KR(X(\R))$ can also be seen as the integral closure of $\R[X]$ in a ring of functions over $X(\R)$. Moreover, we might expect this ring to be the set of functions on $X(\R)$ which are continuous for the Euclidean topology of $X(\C)$. Indeed, the idea of the seminormalization is to glue together the points of the normalization lying over a same point of $X$ and its coordinate ring corresponds to the ring of rational functions which are continuous on $X(\C)$ for the Euclidean topology of $X(\C)$. Since the idea behind the R-seminormalization is to glue together the points of the normalization lying over a same real point of $X$, the continuity on $X(\R)$ for the topology of $X(\C)$ seems to be the right notion to consider. There is, in fact, a little subtlety. The elements of $\KR(X(\R))$ are only defined on $X(\R)$ and they can not necessarily be extended on all $X(\C)$. This is why the correct type of functions to use are the rational functions which satisfies Condition 3 of the next proposition.

\begin{prop}\label{PropCritereSequentiel}
Let $X$ be a real affine variety and let $f : X(\R) \to \R$. Then $f\in\KR(X(\R))$ if and only if it verifies the following properties :
\begin{enumerate}
    \item The function $f$ is rational.
    \item The function $f$ is integral over $\R[X]$.
    \item For all $x\in X(\R)$, for all $(z_n)_n\in \Dom(f)\cup X(\R)$ such that $z_n \to x$ then $f(z_n) \to f(x)$.
\end{enumerate}
\end{prop}

\definecolor{ttqqqq}{rgb}{0.2,0.,0.}
\definecolor{ffqqqq}{rgb}{1.,0.,0.}
\definecolor{xdxdff}{rgb}{0.49019607843137253,0.49019607843137253,1.}
\begin{figure}[h]
\centering
\begin{tikzpicture}[line cap=round,line join=round,>=triangle 45,x=0.6cm,y=0.6cm]
\clip(-12,-2) rectangle (3.9,9);
\draw [line width=2.pt,domain=-13.71950139981348:15.082414948500274] plot(\x,{(-0.-0.*\x)/21.});
\draw [line width=2.pt,color=ffqqqq] (2.48,-7.35479550578517) -- (2.48,11.547897600006063);
\draw (-10,7) node[anchor=north west] {$X(\C)$};
\draw (-12,1) node[anchor=north west] {$X(\R)$};
\begin{scriptsize}
\draw[color=blue] (2.8,-0.3) node {$x$};
\draw [fill=xdxdff] (2.48,0) circle (2.5pt);
\draw[color=ffqqqq] (1.2,8) node {$\Indet(f)$};
\draw [fill=black] (-5.82,2.86) circle (1.0pt);
\draw[color=black] (-5.657720986529486,3.1530806739225934) node {$z_1$};
\draw [fill=black] (-4.78,2.18) circle (1.0pt);
\draw[color=black] (-4.62415939508282,2.4640396129581497) node {$z_2$};
\draw [fill=black] (-3.3,1.52) circle (1.0pt);
\draw[color=black] (-3.1312370963265255,1.7979665873591877) node {$z_3$};
\draw [fill=black] (-1.72,1.) circle (1.0pt);
\draw[color=black] (-1.569410691473786,1.2926698093185958) node {$z_4$};
\draw [fill=black] (-0.22,0.68) circle (1.0pt);
\draw[color=black] (-0.05352035735200936,0.9711173142018555) node {$z_5$};
\draw [fill=ttqqqq] (1.14,0.3) circle (1.0pt);
\draw[color=ttqqqq] (1.3015937292113968,0.580660712988671) node {$z_6$};
\draw [fill=black] (1.92,0.12) circle (1.0pt);
\draw[color=black] (2.0825069316377665,0.3969164300648193) node {$z_7$};
\draw [fill=black] (2.14,0.08) circle (1.0pt);
\draw[color=black] (2.369607373706285,0.43136848311304143) node {};
\draw [fill=black] (2.2698100967769284,0.04852954575141109) circle (1.0pt);
\draw[color=black] (2.4844475505336923,0.38543241238207854) node {};
\end{scriptsize}
\end{tikzpicture}
\caption{Illustration of Condition 3 in Proposition \ref{PropCritereSequentiel}}
\end{figure}
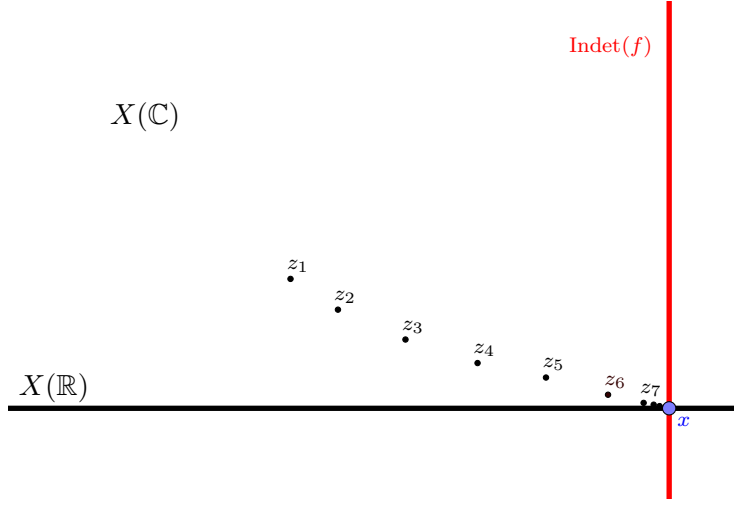

\begin{proof}
Let $f:X(\R) \to \R$ be integral over $\R[X]$ and suppose there exists $p,q\in\R[X]$ such that $f = p/q$ on $\Dcal(q)$. Then there exists $g\in\R[X']$ such that $f\circ \piC' = g$ on $(\piC')^{-1}(\Dcal(q)) = \Dcal(q\circ\piC')$. Note that $f$ can be extended on $\Dcal(q_{\C})\cup X(\R)$. Let $x\in X(\R)$, by Proposition \ref{PropConstanteSurLesFibres} and Lemma \ref{LemR+Condition1et2}, we just have to show that the following propositions are equivalent
\begin{enumerate}
    \item For all $z_1', z_2' \in (\piC')^{-1}(x)$ we have $g_{\C}(z_1')=g_{\C}(z_2') = f(x)$.
    \item For all $(z_n)_n\in \Dcal(q_{\C})\cup X(\R)$ such that $z_n \to x$ then $f(z_n) \to f(x)$.
\end{enumerate}
We prove $1 \implies 2$. Let $(z_n)_n\in \Dcal(q_{\C})\cup X(\R)^{\mathbb{N}}$ be such that $z_n \to x$ and suppose that $\left(f(z_n)\right)_n$ is convergent. We can consider an open ball $B(x,\epsilon) \subset X(\C)$ and an integer $N\in\mathbb{N}$ such that for all $n>N$ we have $z_n\in B(x,\epsilon)$. In particular, it is also true for $B := \overline{B(x,\epsilon)}$ which is compact. By \cite{Bernard2021}, Lemma 4.8, the map $\piC'$ is a proper map and so $(\piC')^{-1}(B)$ is compact. Now, if for all $n>N$ we consider $z_n'\in(\piC')^{-1}(z_n)$, then we obtain a sequence $(z_n')_{n>N}$ whose elements are contained in the compact set $(\piC')^{-1}(B)$. So there exists a convergent subsequence $(z_{n_k}')_{n_k>N}$ and its limit, that we note $l$, belongs to $X'(\C)$ because this set is closed for the Euclidean topology. Moreover, since $\piC'$ is continuous, we have $z_{n_k} = \piC'(z_{n_k}) \to \piC'(l)$. So $\piC'(l) = x$. Then, by continuity of $g$, we have $f(z_{n_k}) = g(z_{n_k}') \to g(l) = f(x)$. This means that the limit of $\left(f(z_n)\right)_n$ is $f(x)$.

Now suppose that $\left(f(z_n)\right)_n$ is not necessarily convergent. Since, for all $n>N$, $f\circ\piC'(z_n') = g(z_n')$, then $\{f(z_n)\}_{n>N} \subset g\left( (\piC')^{-1}(B) \right)$ which is a compact set because $g$ is continuous. Then, by applying the arguments of the preceding paragraph, we can show that every convergent subsequence of $\left(f(z_n)\right)_{n>N}$ admits the same limit which is $f(x)$. So the sequence $\left(f(z_n)\right)_{n>N}$ is convergent and we get $f(z_n)\to f(x)$.

Now we prove $2 \implies 1$. Let $z\in(\piC')^{-1}(x)$. Since $\Dcal(q\circ\piC')$ is dense in $X'(\C)$, we can consider a sequence $(z_n)_n \in \Dcal(q\circ\piC')^{\mathbb{N}}$ such that $z_n \to z$. By continuity of $\piC'$, we have $\piC'(z_n) \to \piC'(z) = x$, so $f(\piC'(z_n)) \to f(\piC'(z))$, by assumption on $f$. The function $g$ being continuous, we have $g(z_n) \to g(z)$ and since for all $n\in\mathbb{N}$, $g(z_n) = f\circ\piC'(z_n)$, we get that $g(z) = f(\piC'(z)) = f(x)$.
\end{proof}

In the case of real curves, the poles of $f$ are small enough for $f$ to be well defined in a neighborhood of $X(\R)$ in $X(\C)$. Thus we get the following corollary.
\begin{cor}
Let $X$ be a real curve. Then 
$$ \KR(X(\R)) = \{ f : X(\R) \to \R \mid f\in\R[X]'\text{ and }f\text{ is continuous on }X(\R)\text{ in }X(\C) \} $$
\end{cor}

\begin{proof}
Let $f\in \KR(X(\R))$. Since $X$ is a curve, then $\Dom(f_{\C}) = X(\C)\setminus\{\text{finite number of points}\}$. So, for all $x\in X(\R)$, there exists $\epsilon > 0$ such that $B_{X(\C)}(x,\epsilon) \cap \Dom(f)^c = x$. By the previous proposition, for all $(z_n)_n \in B_{X(\C)}(x,\epsilon)^{\mathbb{N}}$ such that $z_n \to x$, then $f(z_n) \to f(x)$. So $f$ is continuous at $x$ in $X(\C)$.

\end{proof}

We want now to give a characterization of the elements of $\KR(X(\R))$ with properties on their graphs as in Theorem 4.20 of \cite{Bernard2021}. This will allow us, for instance, to provide non trivial examples of elements of $\KR(X(\R))$. We start by proving two necessary conditions on the graph for a function to be in $\KR(X(\R))$.

\begin{lemma}\label{LemR+Condition3}
Let $X$ be a real affine variety and $f \in \KR(X(\R))$. Then the graph $\Gamma_f$ is Zariski closed in $X(\R)\times\mathbb{A}^1(\R)$.
\end{lemma}

\begin{proof}
Let $f\in \KR(X(\R))$ and $\pi^{R+} : X^{R+} \to X$ be the $\Rtir$seminormalization morphism. By Theorem \ref{TheoRSubEqBij}, we have that $\piR^{R+}$ is bijective. So 
$$ \Gamma_f = \left\{ \left(x;f(x) \right) \mid x\in X(\R)\right\} = \left\{ \left(\piR^{R+}(x);f\circ\piR^{R+}(x) \right) \mid x\in X^{R+}(\R)\right\}$$
But since $f\circ\piR^{R+}$ is a polynomial and, by Proposition \ref{LemFiniBijDoncHomeoSurSpec}, the morphism $\piR^{R+}\times \Id$ is a Z-homeomorphism, we get that
$$ \left\{ \left( \piR^{R+}(x);f\circ\piR^{R+}(x) \right) \mid x\in X^{R+}(\R)\right\} = \piR^{R+}\times Id \left( \Gamma_{f\circ\piR^{R+}} \right)$$
is Z-closed and so that $\Gamma_f$ is a Z-closed subset of $X(\R)\times\A^1(\R)$.
\end{proof}

\begin{notation}
We will denote by $\overline{\Gamma_f}^{\C}$ the Zariski closure of $\Gamma_f$ in the set $X(\C)\times \A^1(\C)$.
\end{notation}

\begin{lemma}\label{LemR+Condition4}
Let $X$ be a real affine variety and $f \in \KR(X(\R))$. Then, for all $x\in X(\R)$, we have $$\overline{\Gamma_f}^{\C} \cap \left( \{x\}\times\A^1(\C) \right) = \Bigl\{\bigl(x;f(x)\bigr)\Bigr\}$$
\end{lemma}

\begin{proof}
Let $f \in \KR(X(\R))$. We can consider $p,q\in \R[X]$ such that $f=p/q$ if $q\neq 0$. Let $(x,t)\in \overline{\Gamma_f}^{\C} \cap \left( \{x\}\times\A^1(\C) \right)$ with $x\in X(\R)$ and consider the Z-dense Z-open set $\Dcal(q) \subset \overline{\Gamma_f}^{\C}$. Since it is also dense for the Euclidean topology, we can consider a sequence $(z_n,t_n)_n \in \Dcal(q)^{\mathbb{N}}$ such that $(z_n,t_n)\to (x,t)$. Moreover, we have $$\overline{\Gamma_f}^{\C} \subset \{q_{\C}t-p_{\C}\} \subset X(\C)\times\A^1(\C)$$
 and so we get, for all $n\in \mathbb{N}$, that $(z_n,t_n) = (z_n,p(z_n)/q(z_n))$. So, by Proposition \ref{PropCritereSequentiel}, we have $$(z_n,t_n) = (z_n,f(z_n)) \to (x,f(x))$$
This means that $\overline{\Gamma_f}^{\C} \cap \Bigl( \{x\}\times\A^1(\C) \Bigr) = \Bigl\{\bigl(x;f(x)\bigr)\Bigr\}$, which conclude the proof.
\end{proof}

We can now show that if an integral rational function on $X(\R)$ verifies Lemma \ref{LemR+Condition3} and \ref{LemR+Condition4}, then it is an element of $\KR(X(\R))$.

\begin{theorem}\label{TheoKRconditions}
Let $X$ be a real affine variety and let $f:X(\R)\to \R$. Then $f\in\KR(X(\R))$ if and only if it verifies the following properties :
\begin{center}
\begin{enumerate}
    \item[1)] $f\in \K(X)$
    \item[2)] There exists a monic polynomial $P(t)\in \R[X][t]$ such that $P(f)=0$ on $X(\R)$.
    \item[3)] The graph $\Gamma_f$ is Zariski closed in $X(\R)\times\mathbb{A}^1(\R)$.
    \item[4)] For all $x\in X(\R)$, we have $\overline{\Gamma_f}^{\C} \cap \left( \{x\}\times\A^1(\C) \right) = \Bigl\{(x;f(x))\Bigr\}$
\end{enumerate}
\end{center}
\end{theorem}

\begin{proof}
Let $f\in\KR(X(\R))$. Then, by Lemma \ref{LemR+Condition1et2}, \ref{LemR+Condition3} and \ref{LemR+Condition4}, it verifies the four properties of the proposition. Conversely, suppose that $f : X(\R)\to \R$ verifies the four properties above. We consider the map
$$\begin{array}{cccc}
     \psi :&\R[X][t]&\to&\K(X)  \\
     &Q(t)& \mapsto & Q(f)
\end{array}$$
and write $\R[Y] \simeq \R[X][t]/\ker\psi \simeq \R[X][f]$ with $\pi:Y\to X$ the morphism induced by $\R[X]\inj \R[Y]$. We then have
$$\R[X] \inj \R[Y] \simeq \R[X][f] \subset \K(X)$$
So $\K(X)\simeq \K(Y)$ and $\pi$ is birational. Moreover $\R[Y]$ is a finite $\R[X]$-module because so is $\R[X][t]/<P(t)>$ and 
$$ \R[Y] \simeq \R[X][t]/\ker\psi \simeq ( \R[X][t]/<P(t)> ) / (\ker\psi / <P(t)> ) $$
Hence $\pi : Y \to X$ is a finite birational morphism. We want to show that the restriction $\widetilde{\pi} : \piC^{-1}(X(\R)) \to X(\R)$ of the morphism $\piC$ is bijective. By assumption, we can consider the real ideal $I_f := \mathcal{I}(\Gamma_f)$ and we have $\Gamma_f = \Zcal(I_f)$. One can see that $I_f \subset \ker\psi$ because $$\forall Q\in I_f\text{\hspace{0.2cm} }\forall x\in X(\R)\text{\hspace{0.3cm} }Q(x,f(x))=0$$
So $$Y(\C) = \Zcal_{\C}(\ker\psi) \subseteq \Zcal_{\C}(I_f) = \overline{\Gamma_f}^{\C} $$
Let $x\in X(\R)$. By the fourth condition, we have 
$$\piC^{-1}(x) = Y(\C) \cap \left( \{x\}\times\A^1(\C) \right) \subseteq \overline{\Gamma_f}^{\C} \cap \left( \{x\}\times\A^1(\C) \right) = \left\{(x;f(x))\right\}$$
and since $\piC : Y(\C) \to X(\C)$ is finite, then it is surjective and we get $\piC^{-1}(x) = \{(x;f(x))\}$. We have shown that $\pi$ is a finite birational morphism and $\widetilde{\pi}$ is bijective. From the universal property of the R-seminormalization, we get
$$ \R[X] \inj \R[Y] \inj \R[X^{R+}] $$
So $f\circ\piR^{R+}\in \R[X^{R+}]$.
\end{proof}

\begin{ex} 
Let $X = \Spec(\R[x,y]/<y^3-x^2y^2+yx^2(x+1)-x^4(x+1)>)$ and consider $$ f = \left\{\begin{array}{l}
    y/x \text{ if }x\neq 0 \\
    0 \text{ else} 
\end{array} \right.$$
We have that $f$ is a root of the polynomial $P_f(t) = t^3-xt^2+t(x+1)-x(x+1)$. Since $0$ is the only real root of $P_f$ when $x=0$, we get that $\Gamma_f$ is a Z-closed subset of $X(\R)\times \A^1(\R)$ given by $$\Gamma_f = \begin{cases}
y^3-x^2y^2+yx^2(x+1)-x^4(x+1) = 0\\
xt-y = 0\\
t^3-xt^2+t(x+1)-x(x+1) = 0
\end{cases}$$

But it doesn't verify the fourth condition because $\overline{\Gamma_f}^{\C} \cap \left( \{0\}\times\A^1(\C) \right) = \left\{(0;0) , (0;\pm i) \right\}$
\end{ex}

\section{Rational functions extending continuously on the complex points}\label{SectionChap3:RealKOcomplex}

In the case of complex varieties, it has been shown in \cite{Bernard2021} that the ring of polynomial functions on the seminormalization corresponds to the ring $\KO(X(\C))$ of rational functions that extend continuously on $X(\C)$. The purpose of this section is to show that, in the case of real varieties, the real valuated functions which are the restriction to $X(\R)$ of an element of $\KO(X(\C))$ correspond to the polynomial functions on the seminormalization. Those functions can also be seen as the elements $f$ in $\KO(X(\C))$ such that $\overline{f(\overline{z})} = f(z)$, for all $z\in X(\C)$. We end this section by giving a characterization of those functions with their graphs, as in Theorem 4.20 of \cite{Bernard2021} or like Theorem \ref{TheoKRconditions}. This will allow us to construct examples of different functions that become polynomials on $X^+$, $X^{R+}$, $X^{s_c}$ or $X^{w_c}$.

\begin{definition}
Let $X$ be a real affine variety, we define the set
$$ \KP(X(\R)) := \{ f : X(\R) \to \R \mid \exists g \in \KO(X(\C)) \text{ such that }g_{|X(\R)} = f \} $$
\end{definition}

\begin{rmq}
By Proposition \ref{PropCritereSequentiel}, it is clear that $\KP(X(\R)) \subset \KR(X(\R))$. Note also that those rings inject in $\K(X)$ and so, even if their elements are defined on all $X(\R)$, they are uniquely represented by the equivalent class of a rational representation.
\end{rmq}

Every real rational function can be extended as a complex rational function. The next lemma says that, for an element of $\KP(X(\R))$, this extension is a rational representation of an element of $\KO(X(\C))$.

\begin{lemma}\label{LemLOdonneRegComplexe}
Let $f:X(\R)\to \R$ and $p,q\in \R[X]$ be such that $f = p/q$ if $q\neq 0$. Then the following statements are equivalent
\begin{enumerate}
    \item The function $f$ belongs to $\KP(X(\R))$.
    \item The rational function $p_{\C}/q_{\C}$ extends continuously to an element $f_{\C}\in \KO(X(\C))$.
\end{enumerate}
\end{lemma}

\begin{proof}
Suppose that $f\in \KP(X(\R))$, then there exists $g = p'/q' \in \KO(X(\C))$ such that $g_{|X(\R)} = f$. Then we have $p'\qC = q'\pC$ on $X(\R)$. So $p'\qC = q'\pC$ on $X(\C)$ because $X(\R)$ is Z-dense in $X(\C)$. So we get $g=f_{\C}$ on $\Dcal(\qC) = \Dcal(q')$ and finally $g=f_{\C}$ on $X(\C)$ by continuity. Conversely, it is clear that if $f_{\C}\in \KO(X(\C))$, then $f = (f_{\C})_{\mid X(\R)} \in \KP(X(\R))$.
\end{proof}

By definition, we know that for each element $f\in \KP(X(\R))$, there is an extension $f_{\C}$ that belongs to $\KO(X(\C))$. One can ask the reverse question : what are the elements of $\KO(X(\C))$ which are the extension of an element of $\KP(X(\R))$ ? The answer is given by the next proposition.

\begin{notation}
Let $E\subset \mathcal{F}(X(\C);\C)$ and $f\in E$. We note $^{\sigma}f$ the function defined by $^{\sigma}f(z) = \overline{f(\overline{z})}$, for all $z\in X(\C)$. We will say that $f$ is $\sigma\textit{-invariant}$ if $^{\sigma}f = f$ and we note $^{\sigma}E$ the elements of $E$ which are $\sigma\textit{-invariant}$. 
\end{notation}

\begin{prop}
Let $X$ be a real algebraic variety. Then we have the following isomorphism 

$$\begin{array}{cccc}
    \psi : &^{\sigma}\KO(X(\C)) & \xrightarrow{\sim} & \KP(X(\R))\\
    &g & \to & g_{\mid X(\R)}
\end{array}$$
\end{prop}

\begin{proof}
First of all, the morphism $\psi$ is well defined because if $g\in ^{\sigma}\KO(X(\C))$, then for all $x\in X(\R)$ we have $ g(x) = {}^{\sigma}g(x) = \overline{g(\overline{x})} = \overline{g(x)}$ and so $g(x) \in \R$.
The morphism $\psi$ is surjective because, by Lemma \ref{LemLOdonneRegComplexe}, if $f \in \KP(X(\R))$, then $f_{\C} \in \KO(X(\C))$ and it is clear that $f_{\C}$ is $\sigma$-invariant.
Let us see now that it is injective. Let $g_1,g_2 \in \KO(X(\C))$ be such that $g_{1_{\mid X(\R)}} = g_{2_{\mid X(\R)}}$. We can write $g_1 = p_1/q_1$ on $\Dcal(q_1)$ and $g_2 = p_2/q_2$ on $\Dcal(q_1)$ for $p_i,q_i\in\C[X]$. Then $q_1q_2g_1 = q_1q_2g_2$ on $X(\R)$, so $p_1q_2 = p_2q_1$ on $X(\R)$. Since $X(\R)$ is Zariski dense in $X(\C)$, the relation extends to $X(\C)$ and so $p_1/q_1 = p_2/q_2$ on $\Dcal(q_1q_2)$. Hence $g_1 = g_2$ by continuity.
\end{proof}

We obtain a real version of Theorem 4.13 of \cite{Bernard2021} saying that the polynomial functions on the seminormalization correspond to the rational functions of $X$ which extend continuously on $X(\C)$.

\begin{cor}
Let $X$ be an real algebraic variety. Then we have the following isomorphism

$$\begin{array}{cccc}
    \varphi : & \KP(X(\R)) & \xrightarrow{\sim} & \R[X^+]\\
    &f & \mapsto & f\circ\piR^+
\end{array}$$
\end{cor}

\begin{proof}
By \cite{Bernard2021}, Theorem 4.13, we have the following isomorphism 
$$\begin{array}{ccc}
    \KO(X(\C)) & \xrightarrow{\sim} & \C[X^+]\\
    f & \mapsto & f\circ\piC^+
\end{array}$$
Let us show that it induces an isomorphism on the $\sigma$-invariant elements
$$\begin{array}{ccc}
    {}^{\sigma}\KO(X(\C)) & \xrightarrow{\sim} & {}^{\sigma}\C[X^+]\\
    f & \mapsto & f\circ\piC^+
\end{array}$$
Let $f\in {}^{\sigma}\KO(X(\C))$ and $z\in X(\C)$, since ${}^{\sigma}\piC^+ = \piC^+$, we get  $\piC^+(\overline{z}) = \overline{{}^{\sigma}\piC^+(z)} =\overline{\piC^+(z)}$ and so 
$$\overline{f\circ\piC^+(\overline{z})} = \overline{f} (\piC^+(\overline{z})) = \overline{f}(\overline{\piC^+(z)}) = {}^{\sigma}f(\piC^+(z)) = f(\piC^+(z))$$
Since ${}^{\sigma}\C[X^+] \simeq \R[X^+]$, we deduce from the previous proposition that 
$$ \begin{array}{ccccccc}
    \KP(X(\R)) & \xrightarrow{\sim} & {}^{\sigma}\KO(X(\C)) & \xrightarrow{\sim} & {}^{\sigma}\C[X^+] & \xrightarrow{\sim} &\R[X^+] \\
    f & \mapsto & f_{\C} & \mapsto & f_{\C}\circ\piC^+ & \mapsto & (f_{\C}\circ\piC^+)_{\mid_{ X(\R)}}
\end{array}$$
One can see that $(f_{\C}\circ\piC^+)_{\mid_{ X(\R)}} = f\circ\piR^+$ so it concludes the proof.
\end{proof}

We now give a characterization of the elements of $\KP(X(\R))$ with their graphs. It can be seen as the real version of Theorem 4.20 of \cite{Bernard2021}.

\begin{lemma}\label{LemRelationEtendue}
Let $X$ be a real affine variety and $f \in \KP(X(\R))$. Let $P\in \R[X][t]$ such that $P(f) = 0$ on $X(\R)$, then $P(f_{\C}) = 0$ on $X(\C)$.
\end{lemma}

\begin{proof}
Let $p$ and $q$ be such that $f=p/q$ on $\Dcal(q)$ and let $P\in \C[X][t]$ be such that $P(f)=0$ on $X(\R)$. In particular $P(p/q)=0$ on $\Dcal(q)$. Then $q_{\C}^{\deg(P)}P(p_{\C}/q_{\C}) = 0$ on $X(\R)$. Since $X(\R)$ is Z-dense in $X(\C)$, we get $q_{\C}^{\deg(P)}P(p_{\C}/q_{\C}) = 0$ on $X(\C)$. Thus $P(f) = P(p_{\C}/q_{\C}) = 0$ on $\Dcal(q_{\C})$ and finally, by continuity, we get $P(f) = 0$ on $X(\C)$.
\end{proof}

\begin{theorem}\label{TheoKPconditions}
Let $X$ be a real affine variety and $f : X(\R) \to \R$, then $f\in \KP(X(\R))$ if and only if it verifies the following properties :
\begin{enumerate}
    \item[1)] $f\in \K(X)$
    \item[2)] There exists a monic polynomial $P\in\R[X][t]$ such that $P(f)=0$ on $X(\R)$.
    \item[3)] The graph $\Gamma_f$ is Zariski closed in $X(\R)\times\A^1(\R)$
    \item[4)] For all $z\in X(\C)$, we have $\#\left( \overline{\Gamma_f}^{\C} \cap \{z\}\times\A^1(\C) \right)= 1$
\end{enumerate}
\end{theorem}

\begin{rmq}
If $f\in \KP(X(\R))$ then the last property says that $\overline{\Gamma_f}^{\C} = \Gamma_{f_{\C}}$.
\end{rmq}

\begin{proof}
Let $f\in \KP(X(\R))$. Then $f\in \KR(X(\R))$ and by Theorem \ref{TheoKRconditions} the function $f$ verifies condition 1, 2 and 3. By Lemma \ref{LemLOdonneRegComplexe}, we have $f_{\C} \in \KO(X(\C))$ and by \cite{Bernard2021} Theorem 4.20, we get that $\Gamma_{f_{\C}}$ is Z-closed. Since $\Gamma_f \subset \Gamma_{f_{\C}}$, we have $\overline{\Gamma_f}^{\C} \subset \Gamma_{f_{\C}}$. Moreover, by Lemma \ref{LemRelationEtendue}, we have that $f_{\C}$ is a complex root of every polynomial defining $\Gamma_f$. So if $I_f$ is a real ideal such that $\Zcal_{\R}(I_f) = \Gamma_f$, then $\Gamma_{f_{\C}} \subset \Zcal_{\C}(I_f) = \overline{\Gamma_f}^{\C}$. We obtain $\Gamma_{f_{\C}} = \overline{\Gamma_f}^{\C}$ and so $f$ verifies Condition 4).

Conversely, let $f : X(\R) \to \R$ be a function that verifies all the four conditions. Thanks to Condition 4), we can define a function $g:X(\C) \to \C$ such that 
$$ \forall z \in X(\C) \text{\hspace{0.2cm} } \Bigl\{ \bigl(z;g(z)\bigr) \Bigr\} = \overline{\Gamma_f}^{\C} \cap \{z\}\times\A^1(\C) $$
Then $\Gamma_g = \overline{\Gamma_f}^{\C}$ so $g$ is rational, integral over $\C[X]$ and $\Gamma_g$ is Z-closed. By \cite{Bernard2021}, Theorem 4.20, this mean that $g\in \KO(X(\C))$. Moreover $\Gamma_f \subset \Gamma_g$, so $g_{|X(\R)} = f$ and we get $f \in \KP(X(\R))$.
\end{proof}

Thanks to Theorems \ref{TheoKRconditions} and \ref{TheoKPconditions}, we give several examples of functions of $\KP(X(\R))$ and $\KR(X(\R))$. In particular, we show that the following inclusions are strict in general.
$$ \R[X] \subset \KP(X(\R)) \subset \KR(X(\R)) \subset \R[X]'_{\RO(Cent(X))} \subset \R[X]'_{\KO(Cent(X))} $$

\begin{ex}
We give an example of a function in $\KP(X(\R))\setminus \R[X]$. Let $X = \Spec(\R[x,y]/<y^2-x^3>)$ and consider the function $$ f = \left\{\begin{array}{l}
    y/x \text{ if }x\neq 0 \\
    0 \text{ else} 
\end{array} \right.$$
We have that $f$ is a root of the polynomial $P(t) = t^2-x$. Since $0$ is the only complex root of $P$ when $x=0$, we get that $\Gamma_f$ is Z-closed in $X(\R)\times \A^1(\R)$ and that $\Gamma_{f_{\C}}$ is Z-closed in $X(\C)\times \A^1(\C)$. So by Theorem \ref{TheoKPconditions}, we have $f\in \KP(X(\R))$.

\end{ex}

\begin{ex}
We give an example of a function in $\KR(X(\R))\setminus \KP(X(\R))$. Consider $X = \Spec(\R[x,y]/<y^2-x^3(x^2+1)^2>)$ and the function $$ f = \left\{\begin{array}{l}
    y/x(x^2+1) \text{ if }x\neq 0 \\
    0 \text{ else} 
\end{array} \right.$$

We have that $f$ is a root of the polynomial $P(t) = t^2-x$. Since $0$ is the only complex root of $P$ when $x=0$, we get that $\Gamma_f$ is Z-closed in $X(\R)\times \A^1(\R)$ and that, for all $x\in X(\R)$, we have $\overline{\Gamma_f}^{\C} \cap \left( \{x\}\times\A^1(\C) \right) = \left\{(x;f(x))\right\}$.
So by Theorem \ref{TheoKRconditions}, we get $f\in \KR(X(\R))$. However, if $x=\pm i$, then $P(t)$ as two distinct complex roots. So $\overline{\Gamma_f}^{\C}$ is not the graph of $f_{\C}$ extended by continuity.\\
\end{ex}

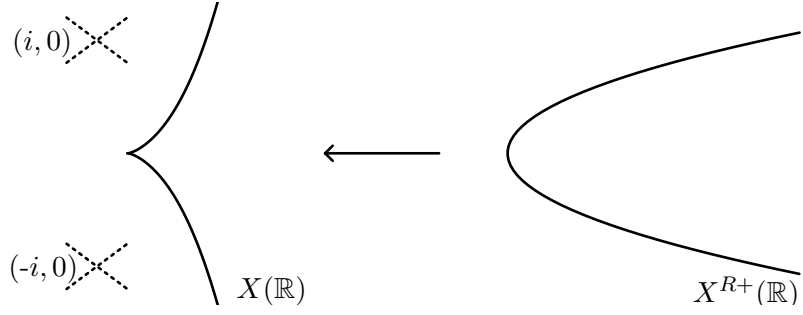
\begin{figure}[h]
    \centering
    \begin{tikzpicture}[line cap=round,line join=round,>=triangle 45,x=1cm,y=1cm]
\clip(-2,-2) rectangle (9,2);
\draw [samples=50,rotate around={-90:(5,0)},xshift=5cm,yshift=0cm,line width=1pt,domain=-1.6:1.6)] plot (\x,{(\x)^2*1.5});
\draw [line width=1pt] (2.6,0)-- (4.1,0);
\draw [line width=1pt] (2.6,0)-- (2.7,0.1);
\draw [line width=1pt] (2.6,0)-- (2.7,-0.1);
\draw (7.3,-1.5) node[anchor=north west] {$X^{R+}(\R)$};
\draw (1.3,-1.5) node[anchor=north west] {$X(\R)$};
\draw (-1.65,1.8) node[anchor=north west] {$(i,0)$};
\draw (-1.7,-1.2) node[anchor=north west] {$(\text{-}i,0)$};

\draw[line width=1pt,smooth,samples=100,domain=0:1.2] plot(\x,{sqrt((\x)^(3)*((\x)^(2)+1))});
\draw[line width=1pt,smooth,samples=100,domain=0:1.2] plot(\x,{0-sqrt((\x)^(3)*((\x)^(2)+1))});

\draw [line width=1pt,dash pattern=on 1pt off 2pt] (-0.8,1.2)--(0,1.8);
\draw [line width=1pt,dash pattern=on 1pt off 2pt] (-0.8,1.8)--(0,1.2);
\draw [line width=1pt,dash pattern=on 1pt off 2pt] (-0.8,-1.2)--(0,-1.8);
\draw [line width=1pt,dash pattern=on 1pt off 2pt] (-0.8,-1.8)--(0,-1.2);
\end{tikzpicture}
    \caption{The R-seminormalization of $\{y^2-x^3(x^2+1)^2=0\}$.}
    \label{fig:ExempleRSN}
\end{figure}

\begin{ex}
We give an example of a function in $\R[X]'_{\RO(Cent(X))}\setminus \KR(X(\R))$. Consider the central variety $X = \Spec(\R[x,y]/<y^4-x(x^2+y^2)>)$ and the function $$ f = \left\{\begin{array}{l}
    y^2/x \text{ if }x\neq 0 \\
    0 \text{ else} 
\end{array} \right.$$
which is a root of the polynomial $P_{(x,y)}(t) = t^2-t-x$. Moreover $f$ is continuous because ?? so $f\in \R[X]'_{\KO(X(\R))}$. Since $X$ is a curve, we have $\KO(X(\R)) = \RO(X(\R))$ and so $f \in \R[X]'_{\RO(X(\R))}$. However, since $P$ has two roots if $x=0$, then the graph $\Gamma_f \varsubsetneq \Zcal(<y^4-x(x^2+y^2);xt-y^2;t^2-t-x>)$ is not Z-closed. By Lemma \ref{LemR+Condition4}, we get that $f \notin \KR(X(\R))$.\\
\end{ex}

\begin{ex}
We give an example of a function in $\R[X]'_{\KO(Cent(X))}\setminus \R[X]'_{\RO(Cent(X))}$. Consider the variety $X = \Spec(\R[x,y]/<x^3-y^3(1+z^2)>)$ and the function $$ f = \left\{\begin{array}{cl}
    x/y & \text{ if }y\neq 0 \\
    ^3\sqrt{1+z^2} & \text{ else} 
\end{array} \right.$$
We have that $f$ is continuous, rational and is a root of the polynomial $P(t) = t^3-1-z^2$. But it is not in $\RO(\Cent(X))$.\\
\end{ex}

To conclude, we summarize this section with the following diagram
$$\xymatrix{
    \R[X] \ar@{^{(}->}[r] & \R[X^+] \ar@{^{(}->}[r] & \R[X^{R+}] \ar@{^{(}->}[r] & \R[X^{s_c}] \ar@{^{(}->}[r] & \R[X^{w_c}] \ar@{^{(}->}[r] & \R[X']\\
    \R[X] \ar@{^{(}->}[r] \ar@{->}[u]^{\simeq} & \KP(X(\R)) \ar@{^{(}->}[r] \ar@{->}[u]^{\simeq} & \KR(X(\R)) \ar@{^{(}->}[r] \ar@{->}[u]^{\simeq} & \R[X]'_{\RO(Cent(X))} \ar@{^{(}->}[r] \ar@{->}[u]^{\simeq} & \R[X]'_{\KO(Cent(X))} \ar@{^{(}->}[r] \ar@{->}[u]^{\simeq} & \R[X]'_{\K(X)} \ar@{->}[u]^{\simeq}
}$$

\section{Biregular normalization and R-seminormalization}\label{SectionLinkR-SNandBirNorm}

Let $A$ be a ring, we say that $A$ satisfies the condition \textbf{(mp)} if $A$ is reduced and has a finite number of minimal primes which are all real ideals. Note that the ring of polynomial functions on an algebraic set always verifies condition \textbf{(mp)}. Moreover, by \cite{IntClosure} Lemma 2.8, if $A$ satisfies \textbf{(mp)} and $B$ is such that $A\inj B\inj A'$, then $B$ satisfy \textbf{(mp)}. For such a ring, we have an injection $$A\inj A/\p_1 \times ... \times A/\p_n \inj \kappa(\p_1)\times ... \times \kappa(\p_n) \simeq \K(A)$$ where the $\p_i$ are the minimal primes of $A$. Remark that the $A/\p_i$ are real rings and so $A$ is also a real ring.\\

In the paper \cite{IntClosure}, the authors introduced the biregular normalization $A^b$ of a ring $A$ which is defined in the following way. Let $A$ be a ring satisfying condition \textbf{(mp)} and let $\mathcal{T}(A)$ be the multiplicative part $1+\sum A^2$ of $A$ which does not contain any zero divisors since $A$ is real. Then we can consider $\Ocal(A) := \mathcal{T}(A)^{-1}A$ and the biregular normalization $A^b$ is define as the integral closure of $A$ in $\Ocal(A)$. This ring can also be defined as the biggest biregular subextension of $A\inj A'$. We recall here Proposition 4.13 from \cite{IntClosure} :

\begin{rappelprop}
Let $A$ be a ring which verifies \textbf{(mp)} and let $B$ such that $A\inj B \inj A'$. Then the following statements are equivalent :
\begin{enumerate}
    \item The extension $A\inj B$ is biregular.
    \item For all $\m \in \Rmax(A)$, there exists a unique $\m' \in \max(B)$ such that $\m' \cap A = \m$. Moreover, the morphism $A_{\m} \to B_{\m'}$ is an isomorphism.
    \item For all $\p \in \Rspec(A)$, there exists a unique $\q \in \Spec(B)$ such that $\q \cap A = \p$. Moreover, the morphism $A_{\p} \to B_{\q}$ is an isomorphism.
\end{enumerate}
\end{rappelprop}

We will say that a ring $A$ is a real affine ring if it is the coordinate ring of a real affine variety. In this case it satisfies the condition \textbf{(mp)} and we have shown Corollary \ref{CorSemiMaxEstSemi} that $A^{R+_{\max}} = A^{R+}$. Moreover, we also have $A^{+_{\max}} = A^+$ by \cite{Bernard2021} Corollary 3.7.

The goal of this section is to compare the notion of biregular normalization, seminormalization and R-seminormalization. We start by giving two lemmas that will lead us to the first comparison.

\begin{lemma}
Let $A\inj B$ be an integral extension of rings with $A$ that satisfy \textbf{(mp)}. If $A\inj B$ is biregular, then $A\inj B$ is $\Rtir$subintegral.
\end{lemma}

\begin{proof}
Let $\p\in\Rspec(A)$, then by \cite{IntClosure} Proposition 4.13 , there exists $\q\in\Spec(B)$ such that $\q\cap A = \p$ and $A_{\p}\simeq B_{\q}$. So we get 
$$\kappa(\p) \simeq A_{\p} / \p A_{\p} \simeq B_{\q}/\q B_{\q} \simeq \kappa(\q)$$
and finally $A\inj B$ is $\Rtir$subintegral.
\end{proof}

\begin{rmq}
As a consequence, we get $A^b \subset A^{R+}$. More precisely, we have the following commutative diagram
$$\xymatrix{
    & A^{+_{\max}} \ar@{^{(}->}[r] & A^{R+_{\max}} \ar@{^{(}->}[r] & A'\\
    A \ar@{^{(}->}[r] \ar@{_{(}->}[rd] & A^+ \ar@{^{(}->}[r] \ar@{->}[u]^{\simeq} & A^{R+} \ar@{->}[u]^{\simeq}&\\
    & A^b \ar@{_{(}->}[ru] &&\\
}$$
\end{rmq}

\begin{lemma}
Let $A\inj B$ be an integral extension of rings and let $A\inj C\inj A^{\F}_B$ be a subextension of the $\F$-seminormalization of $A$ in $B$. Then
$$ A^{\F}_B = C^{\F}_B $$
\end{lemma}

\begin{proof}
By Proposition \ref{PropSnSub}, the extension $C\inj A^{\F}_B$ is $\Ftir$subintegral. So, by the universal property of $C^{\F}_B$, we get $$A\inj C\inj A^{\F}_B\inj C^{\F}_B\inj B$$
and the first three extensions are $\Ftir$subintegral. This implies that $A\inj C^{\F}_B$ is $\Ftir$subintegral and so, by the universal property of $C^{\F}_B$, we obtain $C^{\F}_B\inj A^{\F}_B$.
\end{proof}

By applying this lemma to the subextensions of the previous diagram, we obtain the first comparisons between the considered notions. 

\begin{prop}\label{PropPremiereComparaisons}
Let $A$ be a real affine ring. Then 
\begin{enumerate}
    \item $(A^b)^{R+} = (A^{R+})^b = A^{R+}$
    \item $(A^+)^{R+} = (A^{R+})^+ = A^{R+}$
    \item $(A^b)^+ \subset A^{R+}$
    \item $(A^+)^b \subset A^{R+} $
\end{enumerate}
\end{prop}

\begin{proof}
By applying the previous proposition with $\F = \Rspec$ to the subextension $A\inj A^b\inj A^{R+}$, we get $(A^b)^{R+} =  A^{R+}$. Moreover $ A^{R+} \inj (A^{R+})^b\inj (A^{R+})^{R+}$, so $(A^{R+})^b = A^{R+}$ and 1 is proved. One can do the exact same thing with the subextension $A\inj A^+\inj A^{R+}$ to get 2. The two inclusions 3 and 4 follows because we have
$$A^b \inj (A^b)^+\inj (A^b)^{R+} = A^{R+}$$
and
$$A^+ \inj (A^+)^b\inj (A^+)^{R+} = A^{R+}$$
\end{proof}

We now prove that the inclusion 3 of the previous proposition is an equality. It will be the first part of the main theorem of this section.

\begin{prop}\label{PropAb+EstAR+}
Let $A$ be a real affine ring. Then
$$(A^b)^+ = A^{R+}$$
\end{prop}

\begin{proof}
By Proposition \ref{PropPremiereComparaisons}, we have $(A^b)^+ \inj (A^b)^{R+}$. We want to show that the extension $A^b \inj (A^b)^{R+}$ is subintegral. Since we deal with affine rings, by Theorem \ref{TheoRSubEqBij}, it is enough to show that it is $\Max$-subintegral.

Let $\m^b\in\Max(A^b)$ and let us consider $\m := \m^b\cap A \in \Max(A)$. If $\m\in\Rmax(A)$, because $A\inj A^b$ is biregular, then $\m^b$ is the unique prime ideal of $A^b$ above $\m$ and $\m^b$ is real. But since $A^b \inj (A^b)^{R+}$ is $R$-subintegral, there exists a unique prime ideal of $(A^b)^{R+}$ above $\m^b$.

Now, if $\m\notin\Rmax(A)$, then we have 
$$ (A^b)_{\m} \inj \bigl((A^b)^{R+}\bigr)_{\m} \inj (A^b)'_{\m} = A'_{\m} $$
By \cite{IntClosure} Proposition 4.6, we have $(A^b)_{\m} = A'_{\m}$ so it follows that $(A^b)_{\m} = \bigl((A^b)^{R+}\bigr)_{\m}$. So if we consider $\q_1,\q_2 \in \Max(A^{R+})$ such that $\q_1 \cap A^b = \q_2 \cap A^b = \m^b$, then $\q_1(A^{R+})_{\m} = \q_2(A^{R+})_{\m}$. Suppose there exists $a\in q_1\cap \q_2^c$. Then $\frac{a}{1} \in \q_1(A^{R+})_{\m} = \q_2(A^{R+})_{\m}$. So there exists $b\in \q_2$ and $s\in A\backslash \m$ such that $\frac{a}{1} = \frac{b}{s}$ and so there exists $u\in A\backslash \m$ such that $asu = bu$. We have $a\notin \q_2$ by assumption and also $su\notin \q_2$ or else $su$ would belong to $A\cap \q_2 = \m$. Since $bu \in \q_2$ and $\q_2$ is prime, we obtain a contradiction. This means that $a$ does not exists and so $\q_1 = \q_2$.
\end{proof}

It is shown in $\cite{IntClosure}$ that the biregular normalization of a real variety can be seen as a normalization of its non-real points. We want to understand what the R-seminormalization does locally to the real points and to the non-real points of a variety. This description will also by useful to prove the second part of the main theorem.

\begin{lemma}\label{LemR+loc}
Let $A$ be a real affine ring and $\m\in\Max(A)$. Then 
\begin{enumerate}
    \item If $\m\in\Rmax(A)$, then $(A^{R+})_{\m} = A^+_{\m}$.
    \item If $\m\notin\Rmax(A)$, then $(A^{R+})_{\m} = A_{\m}'$.
\end{enumerate}
\end{lemma}

\begin{proof}
By Proposition \ref{PropAb+EstAR+}, we have $(A^b)^+ = A^{R+}$. So, for $\m\in\Max(A)$, we get $\bigl((A^b)^+\bigr)_{\m} = \bigl(A^{R+}\bigr)_{\m}$ and since the seminormalization commutes with the localization (see \cite{T} for example), we get $\bigl((A^b)^+\bigr)_{\m} = \bigl((A^b)_{\m}\bigr)^+$. By \cite{IntClosure} Proposition 4.6, we have $(A^b)_{\m} = A_{\m}$ if $\m$ is real and $(A^b)_{\m} = A'_{\m}$ else. It follows that $\bigl(A^{R+}\bigr)_{\m} = A_{\m}^+$ if $\m$ is real and $\bigl(A^{R+}\bigr)_{\m} = (A'_{\m})^+ = A'_{\m}$ else.
\end{proof}

\begin{prop}
Let $A$ be a real affine ring. Then
$$ A^{R+} = \bigcap_{\m\in\Rmax(A)} A^+_{\m} \cap \bigcap_{\m\in\Max(A)\setminus \Rmax(A)} A'_{\m}$$
\end{prop}

\begin{proof}
We have
$$ A^{R+} \subset \bigcap_{\m\in\Max(A)} (A^{R+})_{\m} \subset \bigcap_{\m\in\Max(A^{R+})} (A^{R+})_{\m} $$
Since $A^{R+}$ is an affine ring, the last term is equal to $A^{R+}$. So we get $$A^{R+} = \bigcap_{\m\in\Max(A^{R+})} (A^{R+})_{\m}$$
and we conclude with Lemma \ref{LemR+loc}.
\end{proof}

\begin{rmq}
Let $X$ be a real reduced algebraic variety with real irreducible components. From the previous proposition, we see that $X^{R+}$ is the normalization of the non-real locus of $X$ and the seminormalization of its real locus.
\end{rmq}

Using the previous property, we show the second part of the main theorem of this section.

\begin{prop}\label{PropA+bEstR+}
Let $A$ be a real affine ring. Then 
$$ (A^+)^b = A^{R+} $$
\end{prop}

\begin{proof}
By \cite{IntClosure} Proposition 4.8, we have 
$$\begin{array}{llll}
    (A^+)^b & \displaystyle{= \bigcap_{\m\in\Rmax(A^+)} (A^+)_{\m} } & \cap & \displaystyle{ \bigcap_{\m\notin\Rmax(A^+)} (A^+)'_{\m} }\\
     & \displaystyle{ = \bigcap_{\m\in\Rmax(A^+)} \text{ }A^+_{\m} }& \cap & \displaystyle{ \bigcap_{\m\notin\Rmax(A^+)} \text{ }A'_{\m} }
\end{array} $$
By Proposition \ref{PropPremiereComparaisons}, we have $(A^+)^{R+} = A^{R+}$, so 
$$\begin{array}{lllll}
    A^{R+} = (A^+)^{R+} & = \displaystyle{\bigcap_{\m\in\Rmax(A^+)} (A^+)^+_{\m}} & \cap & \displaystyle{\bigcap_{\m\notin \Rmax(A^+)} (A^+)'_{\m} } \\
    & = \displaystyle{\bigcap_{\m\in\Rmax(A^+)} \text{ }A^+_{\m} } & \cap & \displaystyle{\bigcap_{\m\notin \Rmax(A^+)} \text{ }A'_{\m} }
\end{array} $$
Hence $A^{R+} = (A^+)^b$.
\end{proof}

From Proposition \ref{PropA+bEstR+} and Proposition \ref{PropAb+EstAR+}, we get the following theorem.
\begin{theorem}\label{TheoB+est+B}
Let $X$ be a real affine variety. Then
$$ \R[X^{R+}] \simeq \R[X^+]^b \simeq \R[X^b]^+$$
\end{theorem}

We illustrate Theorem \ref{TheoB+est+B} with the following example.

\begin{ex}
Let $X = \Spec\Bigl(\R[x,y]/<(y^4+x^6)(y^2-(x-1)^3(x-2)^2(x^2+1)^2(x^2+4)^3)>\Bigr)$. This variety has at least seven singularities : The real point $(0,0)$ is the crossing of two complex cusps, the real point $(1,0)$ is a non-seminormal point, the point $(2,0)$ is the crossing of two real lines, the two complex conjugate points $(\pm i, 0)$ are the crossing of two complex lines and the two complex conjugate points $(\pm 2i, 0)$ are non-seminormal points. We represent the closed points of $X$,$X^b$,$X^+$,$X^{R+}$ in the following figure.
\end{ex}

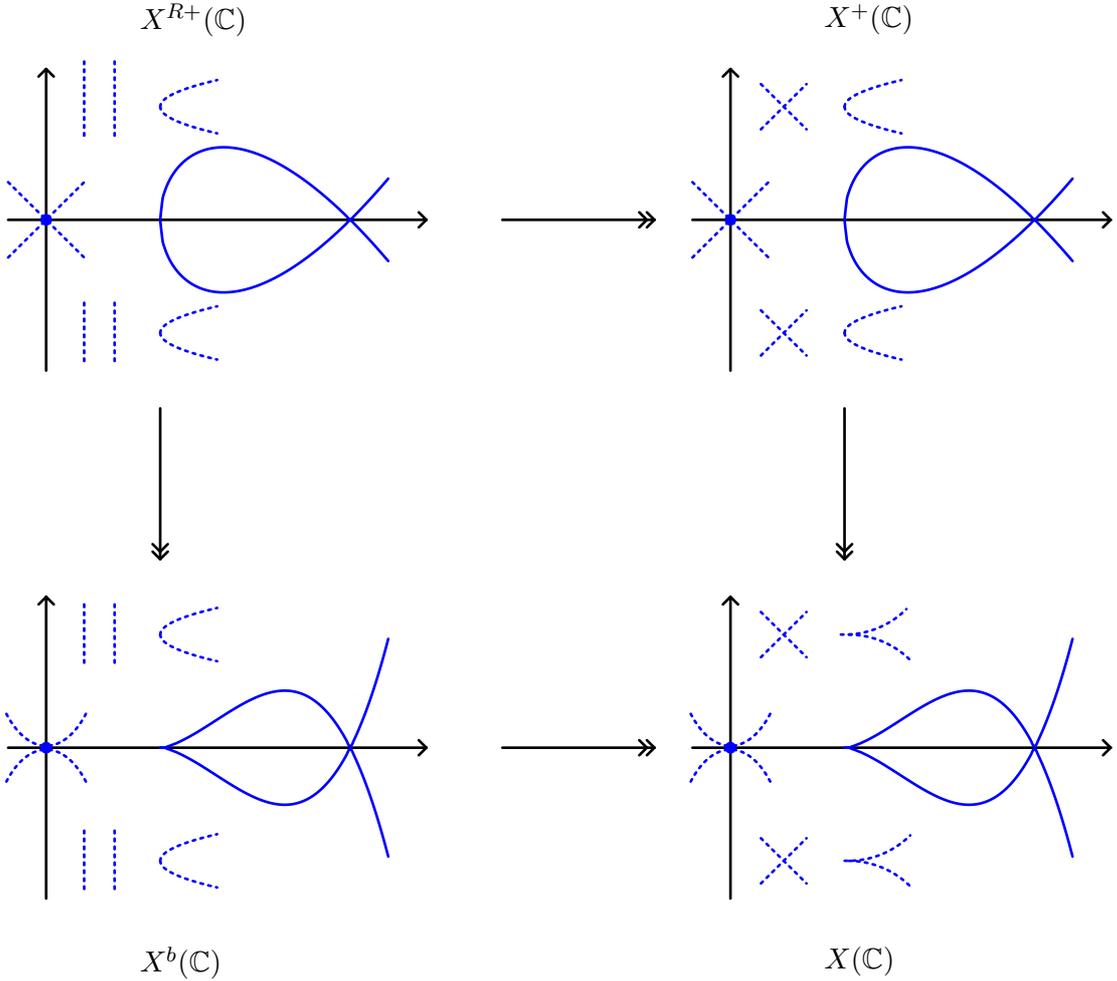
\begin{figure}[h]
    \centering
\begin{tikzpicture}[line cap=round,line join=round,>=triangle 45,x=0.5cm,y=0.5cm]

\clip(-12,-21) rectangle (20,7);

\draw [line width=1pt] (2,0)-- (6,0);
\draw [line width=1pt] (6,0)-- (5.8,0.2);
\draw [line width=1pt] (6,0)-- (5.8,-0.2);
\draw [line width=1pt] (5.8,0)-- (5.6,0.2);
\draw [line width=1pt] (5.8,0)-- (5.6,-0.2);

\draw [line width=1pt] (2,-14)-- (6,-14);
\draw [line width=1pt] (6,-14)-- (5.8,-13.8);
\draw [line width=1pt] (6,-14)-- (5.8,-14.2);
\draw [line width=1pt] (5.8,-14)-- (5.6,-13.8);
\draw [line width=1pt] (5.8,-14)-- (5.6,-14.2);

\draw [line width=1pt] (-7,-5)-- (-7,-9);
\draw [line width=1pt] (-7,-9)-- (-6.8,-8.8);
\draw [line width=1pt] (-7,-9)-- (-7.2,-8.8);
\draw [line width=1pt] (-7,-8.8)-- (-6.8,-8.6);
\draw [line width=1pt] (-7,-8.8)-- (-7.2,-8.6);

\draw [line width=1pt] (11,-5)-- (11,-9);
\draw [line width=1pt] (11,-9)-- (11.2,-8.8);
\draw [line width=1pt] (11,-9)-- (10.8,-8.8);
\draw [line width=1pt] (11,-8.8)-- (11.2,-8.6);
\draw [line width=1pt] (11,-8.8)-- (10.8,-8.6);


\draw [line width=1pt] (8,-10)-- (8,-18);
\draw [line width=1pt] (8,-10)-- (7.8,-10.2);
\draw [line width=1pt] (8,-10)-- (8.2,-10.2);

\draw [line width=1pt] (7,-14)-- (18,-14);
\draw [line width=1pt] (18,-14)-- (17.8,-14.2);
\draw [line width=1pt] (18,-14)-- (17.8,-13.8);

\draw [color=blue, fill=blue] (8,-14) circle (2.0pt);

\draw[color=blue,line width=1pt,smooth,samples=100,domain=11:17] plot(\x,{5*sqrt((\x/5-1.2)^(2)*((\x/5-1.2)-1)^(3)*((\x/5-1.2)-2)^(2))-14});
\draw[color=blue,line width=1pt,smooth,samples=100,domain=11:17] plot(\x,{-5*sqrt((\x/5-1.2)^(2)*((\x/5-1.2)-1)^(3)*((\x/5-1.2)-2)^(2))-14});

\draw [color=blue,line width=1pt,dash pattern=on 1pt off 2pt] (8.8,-11.6)-- (10,-10.4);
\draw [color=blue,line width=1pt,dash pattern=on 1pt off 2pt] (8.8,-10.4)-- (10,-11.6);
\draw [color=blue,line width=1pt,dash pattern=on 1pt off 2pt] (8.8,-17.6)-- (10,-16.4);
\draw [color=blue,line width=1pt,dash pattern=on 1pt off 2pt] (8.8,-16.4)-- (10,-17.6);

\draw [color=blue,shift={(11,-8.7)},line width=1pt,dash pattern=on 1pt off 2pt]  plot[domain=4.67:5.5,variable=\t]({2.3*cos(\t r)},{2.3*sin(\t r)});
\draw [color=blue,shift={(11.1,-13.3)},line width=1pt,dash pattern=on 1pt off 2pt]  plot[domain=0.8:1.6,variable=\t]({2.3*cos(\t r)},{2.3*sin(\t r)});
\draw [color=blue,shift={(11.1,-14.7)},line width=1pt, dash pattern=on 1pt off 2pt]  plot[domain=4.67:5.5,variable=\t]({2.3*cos(\t r)},{2.3*sin(\t r)});
\draw [color=blue,shift={(11.1,-19.3)},line width=1pt, dash pattern=on 1pt off 2pt]  plot[domain=0.8:1.6,variable=\t]({2.3*cos(\t r)},{2.3*sin(\t r)});

\draw [color=blue,shift={(8,-17.05)},line width=1pt,dash pattern=on 1pt off 2pt]  plot[domain=0.8:1.6,variable=\t]({1.5*cos(\t r)},{3*sin(\t r)});
\draw [color=blue,shift={(8,-17.05)},line width=1pt,dash pattern=on 1pt off 2pt]  plot[domain=0.8:1.6,variable=\t]({-1.5*cos(\t r)},{3*sin(\t r)});
\draw [color=blue,shift={(8,-10.95)},line width=1pt,dash pattern=on 1pt off 2pt]  plot[domain=0.8:1.6,variable=\t]({1.5*cos(\t r)},{-3*sin(\t r)});
\draw [color=blue,shift={(8,-10.95)},line width=1pt,dash pattern=on 1pt off 2pt]  plot[domain=0.8:1.6,variable=\t]({-1.5*cos(\t r)},{-3*sin(\t r)});



\draw [line width=1pt] (-10,-10)-- (-10,-18);
\draw [line width=1pt] (-10,-10)-- (-9.8,-10.2);
\draw [line width=1pt] (-10,-10)-- (-10.2,-10.2);

\draw [line width=1pt] (-11,-14)-- (0,-14);
\draw [line width=1pt] (0,-14)-- (-0.2,-14.2);
\draw [line width=1pt] (0,-14)-- (-0.2,-13.8);

\draw [color=blue,fill=blue] (-10,-14) circle (2.0pt);

\draw[color=blue,shift={(-18,0)}, line width=1pt,smooth,samples=100,domain=11:17] plot(\x,{5*sqrt((\x/5-1.2)^(2)*((\x/5-1.2)-1)^(3)*((\x/5-1.2)-2)^(2))-14});
\draw[color=blue,shift={(-18,0)}, line width=1pt,smooth,samples=100,domain=11:17] plot(\x,{-5*sqrt((\x/5-1.2)^(2)*((\x/5-1.2)-1)^(3)*((\x/5-1.2)-2)^(2))-14});

\draw [color=blue,line width=1pt,dash pattern=on 1pt off 2pt] (-9,-10.2) -- (-9,-11.8);
\draw [color=blue,line width=1pt,dash pattern=on 1pt off 2pt] (-8.2,-10.2)-- (-8.2,-11.8);
\draw [color=blue,line width=1pt,dash pattern=on 1pt off 2pt] (-9,-16.2)-- (-9,-17.8);
\draw [color=blue,line width=1pt,dash pattern=on 1pt off 2pt] (-8.2,-16.2)-- (-8.2,-17.8);

\draw[color=blue,shift={(-7,-11)},line width=1pt,dash pattern=on 1pt off 2pt,samples=100,domain=0:1.5] plot(\x,{sqrt((\x/3)});
\draw[color=blue,shift={(-7,-11)},line width=1pt,dash pattern=on 1pt off 2pt,samples=100,domain=0:1.5] plot(\x,{-sqrt((\x/3)});
\draw[color=blue,shift={(-7,-17)},line width=1pt,dash pattern=on 1pt off 2pt,samples=100,domain=0:1.5] plot(\x,{sqrt((\x/3)});
\draw[color=blue,shift={(-7,-17)},line width=1pt,dash pattern=on 1pt off 2pt,samples=100,domain=0:1.5] plot(\x,{-sqrt((\x/3)});

\draw [color=blue,shift={(-10,-17.05)},line width=1pt,dash pattern=on 1pt off 2pt]  plot[domain=0.8:1.6,variable=\t]({1.5*cos(\t r)},{3*sin(\t r)});
\draw [color=blue,shift={(-10,-17.05)},line width=1pt,dash pattern=on 1pt off 2pt]  plot[domain=0.8:1.6,variable=\t]({-1.5*cos(\t r)},{3*sin(\t r)});
\draw [color=blue,shift={(-10,-10.95)},line width=1pt,dash pattern=on 1pt off 2pt]  plot[domain=0.8:1.6,variable=\t]({1.5*cos(\t r)},{-3*sin(\t r)});
\draw [color=blue,shift={(-10,-10.95)},line width=1pt,dash pattern=on 1pt off 2pt]  plot[domain=0.8:1.6,variable=\t]({-1.5*cos(\t r)},{-3*sin(\t r)});



\draw [line width=1pt] (8,4)-- (8,-4);
\draw [line width=1pt] (8,4)-- (7.8,3.8);
\draw [line width=1pt] (8,4)-- (8.2,3.8);

\draw [line width=1pt] (7,0)-- (18,0);
\draw [line width=1pt] (18,0)-- (17.8,-0.2);
\draw [line width=1pt] (18,0)-- (17.8,0.2);

\draw [color=blue,fill=blue] (8,0) circle (2.0pt);

\draw[color=blue,shift={(0,14)},line width=1pt,smooth,samples=100,domain=11:17] plot(\x,{5*sqrt(((\x/5-1.2)-1)*((\x/5-1.2)-2)^(2))-14});
\draw[color=blue,shift={(0,14)},line width=1pt,smooth,samples=100,domain=11:17] plot(\x,{-5*sqrt(((\x/5-1.2)-1)*((\x/5-1.2)-2)^(2))-14});

\draw [color=blue,line width=1pt,dash pattern=on 1pt off 2pt] (8.8,2.4)-- (10,3.6);
\draw [color=blue,line width=1pt,dash pattern=on 1pt off 2pt] (8.8,3.6)-- (10,2.4);
\draw [color=blue,line width=1pt,dash pattern=on 1pt off 2pt] (7,-1)-- (9,1);
\draw [color=blue,line width=1pt,dash pattern=on 1pt off 2pt] (7,1)-- (9,-1);
\draw [color=blue,line width=1pt,dash pattern=on 1pt off 2pt] (8.8,-3.6)-- (10,-2.4);
\draw [color=blue,line width=1pt,dash pattern=on 1pt off 2pt] (8.8,-2.4)-- (10,-3.6);

\draw[color=blue,shift={(11,3)},line width=1pt,dash pattern=on 1pt off 2pt,samples=100,domain=0:1.5] plot(\x,{sqrt((\x/3)});
\draw[color=blue,shift={(11,3)},line width=1pt,dash pattern=on 1pt off 2pt,samples=100,domain=0:1.5] plot(\x,{-sqrt((\x/3)});
\draw[color=blue,shift={(11,-3)},line width=1pt,dash pattern=on 1pt off 2pt,samples=100,domain=0:1.5] plot(\x,{sqrt((\x/3)});
\draw[color=blue,shift={(11,-3)},line width=1pt,dash pattern=on 1pt off 2pt,samples=100,domain=0:1.5] plot(\x,{-sqrt((\x/3)});



\draw [line width=1pt] (-10,4)-- (-10,-4);
\draw [line width=1pt] (-10,4)-- (-10.2,3.8);
\draw [line width=1pt] (-10,4)-- (-9.8,3.8);

\draw [line width=1pt] (-11,0)-- (0,0);
\draw [line width=1pt] (0,0)-- (-0.2,-0.2);
\draw [line width=1pt] (0,0)-- (-0.2,0.2);

\draw [color=blue, fill=blue] (-10,0) circle (2.0pt);

\draw[color=blue,shift={(-18,14)},line width=1pt,smooth,samples=100,domain=11:17] plot(\x,{5*sqrt(((\x/5-1.2)-1)*((\x/5-1.2)-2)^(2))-14});
\draw[color=blue,shift={(-18,14)},line width=1pt,smooth,samples=100,domain=11:17] plot(\x,{-5*sqrt(((\x/5-1.2)-1)*((\x/5-1.2)-2)^(2))-14});

\draw [color=blue,line width=1pt,dash pattern=on 1pt off 2pt] (-11,-1)-- (-9,1);
\draw [color=blue,line width=1pt,dash pattern=on 1pt off 2pt] (-11,1)-- (-9,-1);

\draw [color=blue,line width=1pt,dash pattern=on 1pt off 2pt] (-9,4.2) -- (-9,2.2);
\draw [color=blue,line width=1pt,dash pattern=on 1pt off 2pt] (-8.2,4.2)-- (-8.2,2.2);
\draw [color=blue,line width=1pt,dash pattern=on 1pt off 2pt] (-9,-2.2)-- (-9,-3.8);
\draw [color=blue,line width=1pt,dash pattern=on 1pt off 2pt] (-8.2,-2.2)-- (-8.2,-3.8);

\draw[color=blue,shift={(-7,3)},line width=1pt,dash pattern=on 1pt off 2pt,samples=100,domain=0:1.5] plot(\x,{sqrt((\x/3)});
\draw[color=blue,shift={(-7,3)},line width=1pt,dash pattern=on 1pt off 2pt,samples=100,domain=0:1.5] plot(\x,{-sqrt((\x/3)});
\draw[color=blue,shift={(-7,-3)},line width=1pt,dash pattern=on 1pt off 2pt,samples=100,domain=0:1.5] plot(\x,{sqrt((\x/3)});
\draw[color=blue,shift={(-7,-3)},line width=1pt,dash pattern=on 1pt off 2pt,samples=100,domain=0:1.5] plot(\x,{-sqrt((\x/3)});


\draw (-7.8,-19) node[anchor=north west] {$X^b(\C)$};
\draw (10.2,6) node[anchor=north west] {$X^+(\C)$};
\draw (10.2,-19) node[anchor=north west] {$X(\C)$};
\draw (-7.8,6) node[anchor=north west] {$X^{R+}(\C)$};
\end{tikzpicture}

    \caption{Illustration of Theorem \ref{TheoB+est+B} }
    \label{fig:ExempleFinal}
\end{figure}

\newpage
\bibliographystyle{abbrv}
\bibliography{Biblio.bib}

François Bernard, Université d’Angers, LAREMA, UMR 6093 CNRS, Faculté des Sciences Bâtiment I, 2 Boulevard Lavoisier, F-49045 Angers cedex 01, France\\
\textit{E-mail address: \textbf{francois.bernard@univ-angers.fr}}

\end{document}